\documentclass[12pt]{amsart}
\usepackage{amsmath,amsfonts,mathrsfs,amssymb,cite}
\usepackage{bm}

\newtheorem{thm}{Theorem}[section]
\newtheorem{theorem}[thm]{Theorem}
\newtheorem{corollary}[thm]{Corollary}
\newtheorem{lemma}[thm]{Lemma}
\newtheorem{proposition}[thm]{Proposition}

\theoremstyle{definition}

\theoremstyle{remark}
\newtheorem{remark}[thm]{Remark}

\newenvironment{theorem*}[1]{\smallskip\noindent{\bf #1.}\it}{\medskip}

\numberwithin{equation}{section}

\newcommand\tr{\operatorname{tr}}
\newcommand{\slim}{\operatornamewithlimits{s-lim}}
\newcommand{\re}{{\mathrm e}}
\newcommand{\ri}{{\mathrm i}}
\newcommand{\ls}{\operatorname{ls}}

\newcommand{\bC}{{\mathbb C}}
\newcommand{\bN}{{\mathbb N}}
\newcommand{\bR}{{\mathbb R}}
\newcommand{\bZ}{{\mathbb Z}}

\newcommand\cF{{\mathcal F}}
\newcommand\cI{{\mathcal I}}

\newcommand\cP{{\mathcal P}}

\newcommand\fS{{\mathfrak S}}

\newcommand{\sA}{{\mathscr A}}
\newcommand{\sB}{{\mathscr B}}
\newcommand{\sC}{{\mathscr C}}
\newcommand{\sE}{{\mathscr E}}
\newcommand{\sF}{{\mathscr F}}
\newcommand{\sL}{{\mathscr L}}
\newcommand{\sM}{{\mathscr M}}
\newcommand{\sN}{{\mathscr N}}
\newcommand{\sS}{{\mathscr S}}

\newcommand\al{\alpha}
\newcommand\be{\beta}

\newcommand\la{\lambda}
\newcommand\eps{\varepsilon}
\newcommand\si{\sigma}

\newcommand\bfe{{\mathbf e}}
\newcommand\bx{{\mathbf x}}
\newcommand\by{{\mathbf y}}

\newcommand\bal{\bm{\alpha}}
\newcommand\bbe{\bm{\beta}}

\newcommand\bla{\bm{\lambda}}
\newcommand\bmu{\bm{\mu}}
\newcommand\bnu{\bm{\nu}}
\newcommand\1{\mathbf{1}}

\begin{document}

\title[Inverse spectral problem]{Analyticity and uniform
    stability of the inverse singular Sturm--Liouville spectral problem}%
\author{Rostyslav O.~Hryniv}%
\address{Institute for Applied Problems of Mechanics and Mathematics,
3b~Naukova st., 79601 Lviv, Ukraine \and Institute of Mathematics, the University of Rzesz\'{o}w, 16\,A Rejtana al., 35-959 Rzesz\'{o}w, Poland}%
\email{rhryniv@iapmm.lviv.ua}%

\thanks{}%

\date{\today}%

\begin{abstract}
We prove that the potential of a Sturm--Liouville operator depends analytically and Lipschitz continuously on the spectral
data (two spectra or one spectrum and the corresponding norming
constants). We treat the class of operators with real-valued distributional potentials
in $W^{s-1}_2(0,1)$, $s\in[0,1]$.
\end{abstract}
\maketitle



\section{Introduction}

In this paper we shall establish analyticity and uniform stability of
solutions of two inverse spectral problems for a certain class of
Sturm--Liouville operators on the interval~$[0,1]$. The (direct)
spectral problems to be considered are
\begin{equation}\label{eq:intr.de}
    -y''(x) + q(x)y(x) = \lambda y(x),
\end{equation}
subject to some boundary conditions, where $\la\in\bC$ is the
spectral parameter and~$q$ is a real-valued potential that might be
regular (i.e.\ integrable) or singular (e.g.\ a distribution). For
simplicity, we shall restrict ourselves to the cases of the
Dirichlet boundary conditions
\begin{equation}\label{eq:intr.Dbc}
        y(0)=y(1)=0
\end{equation}
and the Dirichlet--Neumann boundary conditions
\begin{equation}\label{eq:intr.DNbc}
    y(0)=y'(1)=0,
\end{equation}
although other boundary conditions can be treated similarly. (We note
here that the derivative in~\eqref{eq:intr.DNbc} must be replaced with
a quasi-derivative if $q$ is singular, see Section~\ref{sec:pre} for
details.) We shall always denote by
 $\la_1<\la_2<\dots$ the eigenvalues of
problem~\eqref{eq:intr.de}, \eqref{eq:intr.Dbc} and by
 $\mu_1 < \mu_2 < \dots$ those of
problem~\eqref{eq:intr.de}, \eqref{eq:intr.DNbc}.

In 1946, Borg~\cite{Bo} proved that the spectrum of
problem~\eqref{eq:intr.de} corresponding to one fixed set of boundary
conditions, e.g., the Dirichlet~\eqref{eq:intr.Dbc} or the
Dirichlet--Neumann~\eqref{eq:intr.DNbc} ones, does not determine the
potential~$q$ uniquely. (An exceptional situation where only one
spectrum---namely, that for the Neumann boundary conditions---determines the problem was pointed out in 1929
by~Ambartsumyan~\cite{Am}.) However, two such spectra already suffice
to reconstruct the potential~$q$ unambiguously, as follows from the
inverse spectral theory for Sturm--Liouville and Schr\"odinger
operators that emerged in the early 1950-ies from the work of Gelfand
and Levitan~\cite{GL}, Marchenko~\cite{Ma1}, and Krein~\cite{Kr} and
has been extensively developed in many directions since then. This
theory gives an efficient algorithm reconstructing the potential~$q$
from the spectra~$(\la_n)_{n\in\bN}$ and $(\mu_n)_{n\in\bN}$ and also
describes completely the set of possible spectra. For instance, a
typical result~\cite[Thm.~3.4.1]{Ma} reads

\begin{theorem*}{Theorem~A}
For real numbers $\la_1 <\la_2<\dots$ and $\mu_1 < \mu_2 < \dots$ to
give all the Dirichlet and Dirichlet--Neumann eigenvalues of the
Sturm--Liouville problem~\eqref{eq:intr.de} with a real-valued $q\in
L_2(0,1)$, it is necessary and sufficient that these numbers
\begin{itemize}
\item[(i)] interlace, i.e.,
    $\mu_n < \la_n< \la_{n+1}$ for all $n\in\bN$, and
\item[(ii)] have the representation
    \[
        \la_n = \pi^2n^2 + A + a_n,
    \qquad
        \mu_n = \pi^2(n-\tfrac12)^2 + A + b_n,
    \]
    where $A\in\bR$ and $(a_n)_{n\in\bN}$ and $(b_n)_{n\in\bN}$
    are some $\ell_2$-sequences.
\end{itemize}
\end{theorem*}

The induced mapping from the spectral data
$\bigl((\la_n),(\mu_n)\bigr)$ into the potentials~$q$ provides a
solution to the inverse spectral problem and has been extensively
studied in the literature. In particular, this mapping is shown to be
locally continuous in a certain sense, which yields local stability of
the inverse spectral problem, see,
e.g.,~\cite{Al,Bo,FY,Rya1,Rya2,Rya3,McL,Mal,PT,Ho1,Ho2,Ho3,HK,Mi,Ha,MW,Yu} and
the references therein. Here we introduce a metric on the set of the
spectral data $\bigl((\la_n),(\mu_n)\bigr)$ by e.g.\ identifying such
data with the triplets~$\bigl(A,(a_n),(b_n)\bigl)\in\bR\times
\ell_2\times\ell_2$ in the representation of item~(ii) above.
Typically, this local stability states that, for a fixed $M>0$,
there are positive~$\eps$ and $L$ with the following property: if
potentials $q_1$ and $q_2$ are such that $\|q_1\|_{L_p(0,1)}\le M$
and $\|q_2\|_{L_p(0,1)}\le M$ and the corresponding spectral
data $\bnu_1:=\bigl((\la_{1,n}),(\mu_{1,n})\bigr)$ and
$\bnu_2:=\bigl((\la_{2,n}),(\mu_{2,n})\bigr)$ satisfy
$\|\bnu_1-\bnu_2\|\le\eps$, then
\begin{equation}\label{eq:intr.cont}
    \|q_1-q_2\|_{L_p(0,1)} \le L \|\bnu_1-\bnu_2\|
\end{equation}
for a suitable $p\in[1,\infty]$. For instance, local stability results
with $p=2$ were established in~\cite{Rya3,McL} in the regular case
$q\in L_2(0,1)$, and in~\cite{An,CM,McL} for impedance Sturm--Liouville
operators. The cases $p\ge2$ and $p=\infty$ were treated in~\cite{HK} and in~\cite{Mi,Ha} respectively; earlier
Hochstadt in~\cite{Ho2,Ho3} proved stability if only finitely many
eigenvalues in one spectrum are changed. The papers~\cite{Ho1,Rya1}
studied to what extent only finitely many eigenvalues in one or both
spectra determine the potential, and the latter problem in the
non-self-adjoint setting was recently discussed in~\cite{MW}. In~\cite{Mal}, stability of reconstruction of general first-order systems from various given data was investigated. Also,
stability of the inverse spectral problems on semi-axis was proved
in~\cite{MM,Rya2}, and the inverse scattering and inverse resonance problems on the line and half-line were studied in~\cite{DMS,Hi} and \cite{Ko,MSW,MW} respectively.

However, the above results cannot be considered satisfactory, as they
refer to the norm of the potential~$q$ to be recovered and thus
specify neither the allowed noise level~$\eps$ nor the Lipschitz
constant~$L$. Therefore we need a global stability result that
asserts~\eqref{eq:intr.cont} whenever the spectral data $\bnu_1$ and
$\bnu_2$ run through bounded sets~$\sN$ and with $L$ only depending
on~$\sN$.

Such global stability of the inverse problem of reconstructing the
potential~$q$ of a Sturm--Liouville equation from its Dirichlet and
Dirichlet--Neumann spectra was recently established by Savchuk and
Shkalikov in the paper~\cite{SSstab}. In fact, the authors
considered therein the class of problems with distributional~$q$ in the
Sobolev space~$W_2^{s-1}(0,1)$, $s\in(0,1]$. Two typical and most
important examples of singular potentials belonging to~$W^{s-1}_2(0,1)$
only if $s<\tfrac12$ are the Dirac delta-function~$\delta(\cdot-a)$ and
the Coulomb-type interaction~$1/(\cdot-a)$, $a\in(0,1)$. Recently,
singular Sturm--Liouville operators of various types including e.g.\
operators with distributional potentials in~$W_2^{s-1}(0,1)$
~\cite{Sav,SS99,SS03,SSinv,SSas,HMinv,HMscale,HMas} and operators in
impedance form~\cite{An,CM,RS,AHMimp} have attracted considerable
interest, which is partly motivated by their importance for many
applied problems in classical and quantum mechanics, scattering theory
etc. Both the direct and the inverse spectral theory of such singular
operators were developed in detail; we refer the reader to
e.g.~\cite{SS03,SSinv,SSmap,HMinv}.

The purpose of the present paper is two-fold. Firstly, we give an
alternative proof of the global stability of the inverse spectral
problem for Sturm--Liouville operators with distributional potentials
in~$W_2^{s-1}(0,1)$ for $s\in[0,1]$, i.e., including the extreme case
$s=0$ (Theorem~\ref{thm:pre.main}). In fact, we prove that the inverse
spectral mapping is analytic and locally Lipschitz continuous.
Secondly, we prove analogous results for the inverse spectral problem
of reconstructing the potential~$q$ from the corresponding Dirichlet
spectrum~$(\la_n)$ and norming constants~$(\al_n)$ defined in
Section~\ref{sec:pre} (Theorem~\ref{thm:pre.norm}); the local stability
in this setting was established e.g.\ by McLaughlin in~\cite{McL}. Our
approach is completely different from that of the paper~\cite{McL}
(which follows the method developed by P\"oschel and Trubowitz in the
book~\cite{PT}) and of the paper~\cite{SSstab} (which uses the modified
Pr\"ufer angle) and is based on a generalization of the classical
Gelfand--Levitan--Marchenko method as developed e.g. in~\cite{HMinv}.
The proof essentially relies on the fact that certain sets of sines and
cosines form Riesz bases of $L_2(0,1)$ with uniformly bounded upper and
lower bounds~\cite{Hriesz}. As in~\cite{SSstab} we prove first the
required results for the extreme cases~$s=0$ and~$s=1$ and then use the
Tartar nonlinear interpolation~\cite{Ta} to cover the intermediate~$s$.

The paper is organised as follows. In the next section we formulate the
main results and define suitable topologies on the sets of spectral
data. Section~\ref{sec:GLM} describes the reconstruction method based
on the Gelfand--Levitan--Marchenko integral equation.
Theorem~\ref{thm:pre.norm} on analyticity and local Lipschitz
continuity of the inverse spectral mapping using the Dirichlet spectrum
and norming constants is proved in Section~\ref{sec:norm}, and
Section~\ref{sec:two} studies dependence of the norming constants on
two spectra. Finally, Appendices~\ref{sec:RB}, \ref{sec:sobolev},
and~\ref{sec:aux} contain some auxiliary results on Riesz bases,
Sobolev spaces, and Fourier transforms therein.

\emph{Notations.} Throughout the paper, $\sqrt z$ shall denote the
principal branch of the square root that takes positive values
for~$z>0$. For a Hilbert space~$H$, we denote by $\mathscr{B}(H)$ the algebra of bounded linear operators acting in~$H$.


\section{Preliminaries and main results}\label{sec:pre}

In this section we define explicitly the class of Sturm--Liouville
operators to be studied and state the main results for the inverse
problems.

\subsection{The operators}
Given a real-valued distribution $q$ in $W_2^{-1}(0,1)$, we define
Sturm--Liouville operators~$T$ corresponding to the differential
expression
\begin{equation}\label{eq:pre.de}
    -\frac{d^2}{dx^2} + q
\end{equation}
by means of regularisation by quasi-derivatives suggested by Shkalikov
and Savchuk~\cite{SS99,SS03}. We take a real-valued distributional
primitive $\si\in L_2(0,1)$ of~$q$ and set
\[
    l_\si (y) := - (y'-\si y)' - \si y'
\]
for $y\in W^1_2(0,1)$ such that the
\emph{quasi-derivative}~$y^{[1]}:=y'-\si y$ is absolutely continuous
and $l_\si (y)$ is in~$L_2(0,1)$. We then define the operators
$T_{\mathrm D}=T_{\mathrm D}(\si)$ and $T_{\mathrm N}=T_{\mathrm
N}(\si)$ as the restrictions of $l_\si$ onto the functions satisfying
the boundary conditions $y(0)=y(1)=0$ and $y(0)=y^{[1]}(1)=0$
respectively.

Since $l_\si(y) = -y''+qy$ in the sense of distributions, these
operators coincide with the classical Sturm--Liouville operators if
$q\in L_1(0,1)$. Moreover, the operators $T_{\mathrm D}(\si)$ and
$T_{\mathrm N}(\si)$ depend continuously on $\si\in L_2(0,1)$ in the
uniform resolvent sense~\cite{SS99,SS03}. Therefore $T_{\mathrm
D}(\si)$ and $T_{\mathrm N}(\si)$ are the most natural
Sturm--Liouville operators associated with differential
expression~\eqref{eq:pre.de} for $q\in W^{-1}_2(0,1)$.

We observe that although the differential expression~$l_\si$ is
independent of the particular choice of the primitive $\si$ of $q$, the
boundary conditions for the operator~$T_{\mathrm N}(\si)$ and the
norming constants for $T_{\mathrm D}(\si)$ introduced below do depend
on~$\si$ rather than on~$q$. Therefore it is the function~$\si$, and
not $q$, that has to be reconstructed in the inverse spectral problem
for singular Sturm--Liouville operators under consideration. We shall
call $\si$ the \emph{regularized potential} of the Sturm--Liouville
operators $T_{\mathrm D}(\si)$ and $T_{\mathrm N}(\si)$.


\subsection{Spectral data}\label{ssec:sp-data}
It is known~\cite{SS99,SS03} that for a real-valued $\si\in
L_2(0,1)$ the operators~$T_{\mathrm D}(\si)$ and $T_{\mathrm
N}(\si)$ are self-adjoint, bounded below, and have simple discrete
spectra. We denote by $\la_1< \la_2 < \dots$ the eigenvalues of the
operator~$T_{\mathrm D}$ and by $\mu_1<\mu_2< \dots$ those of
$T_{\mathrm N}$ and recall that these eigenvalues interlace, i.e.,
$\mu_n<\la_n< \mu_{n+1}$ for all $n\in\bN$, and satisfy the
relations
\begin{equation}\label{eq:2.sp}
    \sqrt{\la_n} = \pi n + \rho_{2n},
    \qquad
    \sqrt{\mu_n} = \pi (n-\tfrac12) + \rho_{2n-1}
\end{equation}
with some $\ell_2$-sequence $(\rho_n)$~\cite{SS99,SS03}. If~$q$ belongs
to~$W_2^{s-1}(0,1)$ for some~$s\in[0,1]$, then by~\cite{HMas} there
exists a unique function~$\si^*\in W_2^s(0,1)$ such that $\rho_n =
s_n(\si^*)$,  where
\[
    s_n(f) := \int_0^1 f(x) \sin\pi nx\,dx
\]
is the $n$-th sine Fourier coefficient of~$f$.

In the paper~\cite{SSinv}, the mapping
 $\sF_{\textrm{sin}}:\,f \mapsto (s_n(f))_{n=1}^\infty$
defined on the Sobolev spaces~$W_2^s(0,1)$ was studied in detail for
all $s\ge0$. For $s\in [0,1]$ the results of~\cite{SSinv} can be
specified as follows. Denote by $W_{2,0}^1(0,1)$ the subspace of
$W_2^1(0,1)$ consisting of functions that vanish at the endpoints
$x=0$ and $x=1$ and set
\[
    W_{2,0}^s(0,1) := \bigl[ W_{2,0}^1(0,1),L_2(0,1)\bigr ]_s,
        \qquad s\in(0,1),
\]
to be the interpolation space~\cite[Ch.~I.9]{LM}. In particular, for
$s<\tfrac12$ the space $W_{2,0}^s(0,1)$ coincides with $W_2^s(0,1)$, for
$s>\tfrac12$ it is the proper subspace of the latter consisting of functions
vanishing at the endpoints, and $W_{2,0}^{1/2}(0,1)$ is a proper
subspace of $W_2^{1/2}(0,1)$ defined by more complicated conditions. We
also set
\[
    \ell_2^s:= \{ \bx = (x_n)_{n=1}^\infty
                \mid \|\bx\|^2_s:= \sum n^{2s}|x_n|^2<\infty\};
\]
this is a Hilbert space under the scalar product
 \(
    (\bx,\by)_s := \sum n^{2s} x_n \overline{y_n}
 \)
for $\bx:=(x_n)_{n=1}^\infty$ and  $\by:=(y_n)_{n=1}^\infty$.
The interpolation theory then shows that for all $s\in[0,1]$ the
mapping $\sF_{\mathrm{sin}}$ is an isomorphism between
$W_{2,0}^s(0,1)$ and $\ell_2^s$; moreover, under an equivalent norm
on $W_{2,0}^s(0,1)$ this mapping becomes isometric.

We set $P_0(x) := 1-x$ and $P_1(x)= x$; then for an arbitrary $f\in W_2^1(0,1)$ the function
\[
    f_0(x) := f(x) - f(0) P_0(x) - f(1) P_1(x)
\]
belongs to~$W_{2,0}^1(0,1)$;
thus $\sF_{\mathrm{sin}}(f) \in \ell_2^1 \dotplus \ls\{\bfe_0,\bfe_1\}$, with
\[
    \bfe_j:= \sF_{\mathrm{sin}}(P_j)
    = \bigl( (-1)^{j(n-1)} /(\pi n) \bigr)_{n=1}^\infty
\]
and $\ls S$ standing for the linear span of a set~$S$.
We now set $\hat \ell_2^s := \ell_2^s$ for $s\in[0,\tfrac12)$ and
\[
    \hat \ell_2^s := \ell_2^s \dotplus \ls\{\bfe_0,\bfe_1\}
\]
for $s\in[\tfrac12,1]$; in this latter case the scalar product
in~$\hat\ell_2^s$ is introduced via
\[
     \bigl( \bx+ a_0\bfe_0+a_1\bfe_1, \by+ b_0\bfe_0+b_1\bfe_1
     \bigr)_s
        := (\bx,\by)_s + a_0\overline{b_0} + a_1\overline{b_1}.
\]
Then by Lemma~1 of~\cite{SSinv} we see that $\sF_{\mathrm{sin}}$
extends to an isomorphism of $W_2^s(0,1)$ and $\hat \ell_2^s$ for
all $s\in[0,1]$, which, moreover, is even isometric under an
equivalent norm on~$W_2^s(0,1)$. Therefore, for $q\in W_2^{s-1}(0,1)$ the sequence~$(\rho_n)$ defined by~\eqref{eq:2.sp} belongs to~$\hat \ell_2^s$.

Next we define the norming constants. For a nonzero $z\in\bC$, we
denote by $y(\cdot,z)$ a solution of the equation $l_\si(y) = z^2 y$
satisfying the initial conditions~$y(0)=0$ and $y^{[1]}(0)=z$. Then
$y(\cdot,\sqrt{\la_n})$ is an eigenfunction corresponding to the
eigenvalue~$\la_n$ of the operator~$T_{\mathrm D}$, and we call the
number $\al_n := \|y(\cdot,\sqrt{\la_n})\|_{L_2}^{-2}$ the
\emph{norming constant} for this eigenvalue. It is known~\cite{HMscale}
that if $q\in W_2^{s-1}(0,1)$ for $s\in [0,1]$, then
\[
    \al_n = 2 + \beta_{2n},
\]
where $\beta_{2n}$ is the $2n$-th cosine Fourier coefficient
$c_{2n}(\tilde \si)$ for a unique function $\tilde \si\in
W_2^s(0,1)$ of zero mean that is even with respect to $x=\tfrac12$, i.e.,
$\tilde\sigma(1-x)= \tilde\sigma(x)$; here we set
\[
    c_n(f) := \int_0^1 f(x) \cos\pi nx\,dx.
\]
By the arguments similar to the above the mapping
 $\sF_{\mathrm{cos}}:\, f
    \mapsto \bigl( c_{n}(f)\bigr)_{n=1}^\infty$
is an isomorphism between the subspace
 $\widetilde W_{2,\mathrm{even}}^s(0,1)$
of even (with respect to $x=\tfrac12$) functions in $W_2^s(0,1)$ of zero
mean and the subspace $\ell_{2,\mathrm{even}}^s$ of sequences in
$\ell_2^s$ with vanishing odd entries. Clearly, the spaces~$\ell_{2,\mathrm{even}}^s$ and $\ell_2^s$ are isomorphic.

The norming constants $\al_n$ can be determined from the spectra of the
operators~$T_{\mathrm D}$ and $T_{\mathrm N}$ as follows. We set
$S(z):= y(1,z)$ and $C(z):=y^{[1]}(1,z)$. Due to the integral representations~\eqref{eq:pre.y} and
\eqref{eq:pre.y'} below, $S$ and $C$ are entire functions of exponential
type~$1$ with zeros $0,\pm \sqrt{\la_n}$ and $0,\pm \sqrt{\mu_n}$
respectively. The Hadamard canonical products of~$S$ and $C$ are
\begin{equation}\label{eq:pre.PhiPsi}
    S(z) = z \prod_{n=1}^\infty \frac{\la_n - z^2}{\pi^2n^2},
    \qquad
    C(z) = z \prod_{n=1}^\infty
        \frac{\mu_n-z^2}{\pi^2(n-\tfrac12)^2},
\end{equation}
so that $S$ and $C$ are uniquely determined by their zeros. Finally, we
have~\cite{HMtwo}
\begin{equation}\label{eq:pre.al}
    \al_n = \frac{2\sqrt{\la_n}}{\dot{S}(\sqrt{\la_n})C(\sqrt{\la_n})},
\end{equation}
where the dot denotes the derivative in~$z$.

\subsection{Main results}
Without loss of generality, we may consider only uniformly positive
operators, adding the number $\mu_1 + 1$ to $q$ and all the Dirichlet
and Dirichlet--Neumann eigenvalues as required. Respectively, we
introduce the set~$\sN^s$ of data
$\bigl((\la_n)_{n\in\bN},(\mu_n)_{n\in\bN}\bigr)$ with the following
properties:
\begin{itemize}
\item [(N1)] $\mu_1\ge1$ and the sequences $(\la_n)$ and $(\mu_n)$ strictly interlace,
    i.e., $\mu_n<\la_n<\mu_{n+1}$ for all $n\in\bN$;
\item [(N2)] the numbers $\rho_{2n}:= \sqrt{\la_n} - \pi n$ and
    $\rho_{2n-1}:=\sqrt{\mu_n}- \pi(n-\tfrac12)$, $n\in\bN$, form a sequence
    $(\rho_n)_{n\in\bN}$ in $\hat \ell_2^s$.
\end{itemize}
In this way every element $\bnu:=(\bla,\bmu)$~of $\sN^s$ is
identified with a sequence $(\rho_n)$ in $\hat \ell_2^s$ or,
equivalently, with a unique function~$f\in W_2^s(0,1)$
satisfying~$s_n(f) = \rho_n$. This induces a topology on~$\sN^s$;
moreover, $\|(\rho_n)\|_s$ or $\|f\|_{W_2^s}$ define equivalent
metrics on~$\sN^s$.

According to~\cite{HMscale}, every element of~$\sN^s$ consists of
eigenvalue sequences of the operators $T_{\mathrm D}(\si)$ and
$T_{\mathrm N}(\si)$ corresponding to some real-valued regularized
potential~$\si\in W_2^{s}(0,1)$ and, conversely, for every real-valued
$\si\in W_2^{s}(0,1)$ with $T_{\mathrm N}(\si)\ge I$ the corresponding
spectral data form an element of~$\sN^s$. When the regularized
potential~$\si$ varies over a bounded subset of~$W_2^s(0,1)$, then the
main theorem of the paper~\cite{SSas} implies that the corresponding
spectral data~$\bigl((\la_n),(\mu_n)\bigr)$ remain in a bounded subset
of $\sN^s$. Moreover, the Pr\"ufer angle technique used in~\cite{SSas}
yields then a positive $h$ such that all the corresponding spectral
data~$\bigl((\la_n),(\mu_n)\bigr)$ are $h$-\emph{separated}, i.e., such
that the inequalities $\sqrt{\mu_{n+1}}-\sqrt{\la_n}\ge h$ and
$\sqrt{\la_n}-\sqrt{\mu_n}\ge h$ hold for every $n\in\bN$. Summarizing,
we conclude that the uniform stability of the inverse spectral problem
we would like to establish is only possible on the convex closed
sets~$\sN^s(h,r)$ of the spectral data consisting of all elements
of~$\sN^s$ that are $h$-separated and satisfy $\|(\rho_n)\|_s\le r$.

In these notations, one of the main results of the paper reads as follows.

\begin{theorem}\label{thm:pre.main}
For every $s\in[0,1]$, $h\in(0,\pi/2)$, and $r>0$, the
mapping
\[
    \sN^s(h,r) \ni \bnu \mapsto \si \in W_2^s(0,1)
\]
is analytic and Lipschitz continuous.
\end{theorem}

Lipschitz continuity means that there exists a number~$L=L(s,h,r)$
such that for any two elements $\bnu_1$ and $\bnu_2$
of~$\sN^s(h,r)$ the regularized potentials~$\si_1$ and $\si_2$
in~$W_2^{s}(0,1)$ solving the inverse spectral problems for the
data~$\bnu_1$ and $\bnu_2$ satisfy
\[
    \|\si_1-\si_2\|_{W_2^{s}(0,1)}
        \le L \|\bnu_1-\bnu_2\|_{\sN^s}.
\]
See \cite{Di} for definitions and properties of analytic mappings
between Banach spaces.

In fact, we prove first analyticity and local Lipschitz continuity in the inverse
spectral problem of reconstructing $\si$ from the spectrum $(\la_n)$
of $T_{\mathrm D}(\si)$ and the norming constants $(\al_n)$ (see
Theorem~\ref{thm:pre.norm} below), and then derive
Theorem~\ref{thm:pre.main} by showing that the norming constants
depend analytically and locally Lipschitz continuously on the two spectra.

More exactly, we denote by $\sL^s$ the family of strictly increasing
sequences $\bla:=(\la_n)$ for which the sequence~$(\rho_n)$ with
$\rho_{2n-1}=0$ and $\rho_{2n}:=\sqrt{\la_n}-\pi n$ forms an element
of~$\hat\ell_{2}^s$ and introduce the topology on~$\sL^s$
by identifying such~$\bla$ with~$(\rho_n)\in\hat\ell_{2}^s$.
It follows from the results of Subsection~\ref{ssec:sp-data} that the subspace
$\hat\ell_{2,\mathrm{even}}^s$ of $\hat \ell_2^s$ consisting of elements with vanishing odd entries and the subspace
$W_{2,\mathrm{odd}}^s(0,1)$ of functions in~$W_2^s(0,1)$ which are odd
with respect to~$x=\tfrac12$ are isomorphic under the sine Fourier
transform~$\sF_{\mathrm{sin}}$. For $h\in(0,\pi)$ and $r>0$, we denote
by~$\sL^s(h,r)$ the closed convex subset of $\sL^s$ consisting of
sequences $(\la_n)$ with $\la_1\ge1$, $\sqrt{\la_{n+1}}-\sqrt{\la_n}\ge
h$, and such that $\|(\rho_n)\|_s\le r$.

Next, we write $\sA^s$ for the set of sequences $(\al_n)_{n\in\bN}$ of
positive numbers for which the sequence~$(\be_n)_{n\in\bN}$, with
$\be_{2n-1}=0$ and $\be_{2n}:=\al_n-2$, belongs
to~$\ell_{2,\mathrm{even}}^s$. This induces the topology of~$\ell_2^s$
on $\sA^s$; by the results of Subsection~\ref{ssec:sp-data} the space
$\ell_{2,\mathrm{even}}^s$ (and thus~$\sA^s$) is homeomorphic to the
subspace~$\widetilde W_{2,\mathrm{even}}^s(0,1)$ of functions in~$W_2^s(0,1)$ of zero mean
which are even with respect to~$x=\tfrac12$. We further consider closed
subsets $\sA^s(h,r)$ of $\sA^s$ consisting of all $(\al_n)$ satisfying the inequality~$\al_n \ge h$ for all~$n\in\bN$ and such that $\|(\beta_n)\|_s \le r$.

It follows from~\cite{HMscale} that to every $(\bla,\bal)\in\sL^s\times
\sA^s$ there corresponds a unique regularized real-valued potential
$\si\in W_2^s(0,1)$ such that $\bla$ and $\bal$ are the sequences of
eigenvalues and the corresponding norming constants of the operator~$T_{\mathrm{D}}(\sigma)$. The more elaborate
properties of this mapping are given by the following theorem.

\begin{theorem}\label{thm:pre.norm}
For every $s\in[0,1]$, $h\in(0,\pi)$, $h'\in(0,2)$, and any
positive $r$ and~$r'$, the inverse spectral mapping
\[
    \sL^s(h,r)\times\sA^s(h',r') \ni \bigl(\bla,\bal\bigr)
        \mapsto \si \in W_2^s(0,1)
\]
is analytic and Lipschitz continuous.
\end{theorem}

We conclude this section by observing that analyticity and other
properties of the direct spectral mapping
\[
    W_2^s(0,1) \ni \si \mapsto (\bla,\bmu,\bal) \in \sN^s \times \sA^s,
\]
at least for the classical case $s=1$, are well known and studied in
detail in, e.g.,~\cite{PT}.

%
%

\section{Solution of the inverse spectral problem via the GLM
equation}\label{sec:GLM}

In this section, we recall shortly the method of reconstructing the regularized
potential~$\si$ based on the Gelfand--Levitan--Marchenko (GLM)
equation~\cite{HMinv}; see also~\cite{SSinv} for an alternative approach. Using this method, we shall later show analyticity and
Lipschitz continuity of the inverse spectral mapping.

We recall that $y(\cdot,z)$ stands for the solution of the
equation $l_\si(y) = z^2 y$ satisfying the initial
conditions~$y(0,z)=0$ and $y^{[1]}(0,z)=z$. This function has the
representation~\cite{HMtr}
\begin{equation}\label{eq:pre.y}
    y(x,z) = \sin z x + \int_0^x k(x,t) \sin z (1-2t)\,dt,
\end{equation}
where the kernel $k$ is lower-triangular (i.e., $k(x,t)=0$ a.e.\ in the
domain $0\le x\le t\le1$) and has the property that, for every fixed
$x\in[0,1]$, the functions $k(\cdot, x)$ and $k(x,\cdot)$ are
in~$L_2(0,1)$ and depend continuously therein on~$x\in[0,1]$. Also,
\begin{equation}\label{eq:pre.y'}
    y^{[1]}(x,z) = z\cos z x+ z \int_0^x k_1(x,t)\cos z (1-2t)\,dt
\end{equation}
where a kernel~$k_1$ has similar properties. Denoting by~$K$ the
integral operator with kernel~$k$ and by $I$ the identity operator
in~$L_2(0,1)$, we see that $I+K$ is the transformation operator mapping
solutions of the unperturbed ($\si=0$) differential equation~$l_0(y) =
z^2 y$ to those of~$l_\si(y)=z^2y$.

The GLM equation relates the spectral data for the
operator~$T_{\mathrm D}=T_{\mathrm D}(\si)$ (i.e., its eigenvalues
and norming constants) with the transformation operator $I+K$. To
derive it, we start with the resolution of identity for $T_{\mathrm
D}$,
\[
    I = \slim_{N\to\infty} \sum_{n=1}^N \al_n (\,\cdot\,, y_n) y_n,
\]
where $\slim$ stands for the limit in the strong operator topology, $(\cdot,\cdot)$ is the scalar product in~$L_2(0,1)$, and
$y_n:= y (\cdot,\sqrt{\la_n})$. Recalling that $y_n = (I+K)s_n$ with
$s_n(x) = \sin \sqrt{\la_n} x$, we get
\begin{equation}\label{eq:pre.I}
    I = (I+K)
    \Bigl[\slim_{N\to\infty} \sum_{n=1}^N \al_n (\,\cdot\,,s_n) s_n\Bigr]
        (I+K^*),
\end{equation}
The operator in the square brackets has the form $I+F$, where $F=F(\bla,\bal)$ is an integral operator of Hilbert--Schmidt class~$\fS_2$ with kernel
\begin{equation}\label{eq:pre.f}
        f(x,t):= \tfrac12\bigl(\phi(x+t) - \phi(|x-t|)\bigr),
\end{equation}
where
\begin{equation}\label{eq:pre.phi}
   \phi(x) := \sum_{n\in\bN}
        \bigl( 2\cos \pi nx - \al_k \cos\sqrt{\la_n}x\bigr)
\end{equation}
is a function in~$L_2(0,2)$. Applying $(I+K^*)^{-1}$ to both sides
of~\eqref{eq:pre.I} and rewriting the resulting relation in terms of
the kernels $k$ and $f$, we get the GLM equation
\begin{equation}\label{eq:pre.GLM}
    k(x,t) + f(x,t) + \int_0^x k(x,\xi) f(\xi,t)\,d\xi = 0,
    \qquad x\ge t.
\end{equation}

The algorithm of reconstructing $q$ from the spectral data
$\bigl((\la_n),(\mu_n)\bigr)$ proceeds now as follows. We first
calculate the numbers~$\al_n$ via~\eqref{eq:pre.al}, then construct
the function~$\phi$ of~\eqref{eq:pre.phi}, form the kernel $f$
of~\eqref{eq:pre.f}, solve the GLM equation~\eqref{eq:pre.GLM}
for~$k$, and set
\begin{equation}\label{eq:pre.sigma}
    \si(x) := -\phi(2x) - 2 \int_0^x k(x,\xi) f(\xi,x)\, d\xi.
\end{equation}
Then $\si\in L_2(0,1)$ is the unique regularized potential for which
the operators $T_{\mathrm D}(\si)$ and $T_{\mathrm N}(\si)$ have
eigenvalues $\la_n$ and $\mu_n$, $n\in\bN$, respectively, see~\cite{HMinv,HMtwo}. Moreover, if $k$ and $\phi$ are smooth, then the GLM equation
implies that \(
    \si(x) = 2 k(x,x) - \phi(0),
\) thus yielding the classical relation
\[
    q(x) = 2 \frac{d}{dx} k(x,x)
\]
for the potential $q$. It was proved in~\cite{HMscale} that if the
sequences $(\la_n)$ and $(\mu_n)$ are such that the corresponding
sequence $(\rho_n)$ is in $\hat\ell_2^s$, with $s\in(0,1]$, then
$\si$ of~\eqref{eq:pre.sigma} belongs to~$W_2^s(0,1)$, i.e., the
reconstructed potential $q$ is in $W^{s-1}_2(0,1)$.

%
%

\section{Reconstruction from a spectrum and norming
constants}\label{sec:norm}

In this section, we prove Theorem~\ref{thm:pre.norm} on analyticity and Lipschitz continuity in the inverse spectral problem of reconstructing the
regularized potentials of Sturm--Liouville differential expressions
from their Dirichlet spectra and the corresponding norming constants.

We shall study the correspondence between the data
$(\bla,\bal)\in\sL^s(h,r)\times \sA^s(h',r')$ and the regularized
potentials $\si\in W_2^{s}(0,1)$ of the Sturm--Liouville operator
$T_{\mathrm D}(\si)$  through the chain of mappings
\[
    (\bla,\bal) \mapsto \phi
    \mapsto F \mapsto K \mapsto \si,
\]
in which $\phi\in L_2(0,2)$ is the function of~\eqref{eq:pre.phi},
$F$ is the operator with kernel~$f$ of~\eqref{eq:pre.f}, $K\in\fS_2$
is the integral operator with kernel~$k$ that solves the GLM
equation~\eqref{eq:pre.GLM}, and, finally, $\si$ is given
by~\eqref{eq:pre.sigma}. Throughout this section, we fix $h\in(0,\pi)$, $h'\in(0,2)$, and positive $r$ and $r'$, and shall always denote by $\lambda_n$ the elements of a sequence $\bla\in\sL^s(h,r)$ and by $\al_n$ those
of an~$\bal\in\sA^s(h',r')$.



\begin{lemma}\label{lem:phicont}
The mapping
\[
    \sL^s(h,r)\times \sA^s(h',r') \ni (\bla,\bal)
            \mapsto \phi\in W^s_2(0,2)
\]
is analytic and Lipschitz continuous.
\end{lemma}

\begin{proof}
We have
 \(
     \phi(2x) = \varphi_{\bla}(x) + \psi_{\bla,\bal}(x),
 \)
where
\[
    \varphi_{\bla}(x)
        := 2\sum_{n=1}^\infty[\cos (2\pi nx) -\cos(2\sqrt{\la_n}x)]
\]
and
\[
    \psi_{\bla,\bal}(x):= -\sum_{n=1}^\infty \beta_n \cos(2\sqrt{\la_n}x),
\]
with $\be_n:=\al_n-2$. We recall that $\sqrt\la_n = \pi n + s_{2n}(f)$
for a unique $f\in W_{2,\mathrm{odd}}^s(0,1)$ and $\be_n = c_{2n}(g)$
for a unique $g\in \widetilde W_{2,\mathrm{even}}^s(0,1)$ and that
the mapping $(\bla,\bal) \mapsto (f,g)$ is isomorphic from $\sL^s\times
\sA^s$ into $W_{2,\mathrm{odd}}^s(0,1)\times
\widetilde W_{2,\mathrm{even}}^s(0,1)$. Therefore analyticity and Lipschitz
continuity of the mapping under consideration follows from
Corollary~\ref{cor:C.Phi}.
\end{proof}

It is advantageous to regard the GLM equation~\eqref{eq:pre.GLM} as
the relation between the integral operators $K$ and $F$ generated by
the kernels $k$ and $f$. We shall need several related notions,
which we now recall.

The ideal $\fS_2$ of Hilbert--Schmidt operators consists of integral
operators with square summable kernels, and the scalar product $\langle
X,Y\rangle_2 := \tr(X Y^*)$ introduces a Hilbert space structure on
$\fS_2$. For an integral operator~$T\in\fS_2$ with kernel~$t$ we find that
    $\langle T,T\rangle_2 = \int_0^1\int_0^1 |t(x,y)|^2\,dx\,dy$;
thus the estimate
\begin{equation}\label{eq:FS2}
\begin{aligned}
  \int_0^1 |f(x,y)|^2\,dx
    &\le \tfrac12\int_0^1 |\phi(x+y)|^2\,dx
      +  \tfrac12\int_0^1 |\phi(|x-y|)|^2\,dx \\
     &\le \int_0^{2} |\phi(\xi)|^2\,d\xi
\end{aligned}
\end{equation}
shows that $F\in \fS_2$ and
$\|F\|^2_{\fS_2}=\langle F,F\rangle_2\le \|\phi\|^2_{L_2(0,2)}$.
Moreover, the mapping $\phi\mapsto F$ is linear
(and thus analytic and Lipschitz continuous) from $L_2(0,2)$ into $\fS_2$.

Further, we denote by $\fS_2^+$ the subspace of $\fS_2$ consisting
of all Hilbert--Schmidt operators with lower-triangular kernels. In
other words, $T \in \fS_2$ belongs to $\fS_2^+$ if the kernel $t$
of~$T$ satisfies $t(x,y) = 0$ a.e.\ outside the domain
$\Omega^+:=\{(x,y)\mid 0<y<x<1\}$. For an arbitrary $T\in\fS_2$ with
kernel $t$ the cut-off $t^+$ of $t$ given by
\[
    t^+(x,y) = \left\{ \begin{array}{ll}
        t(x,y)& \quad \mbox {for $x \ge y$},\\
              0 & \quad \mbox {for $x < y$}, \end{array}
        \right.
\]
generates an operator $T^+\in\fS_2^+$, and the corresponding mapping
$\cP^+: T \mapsto T^+$  turns out to be an orthoprojector in $\fS_2$
onto $\fS_2^+$, i.e. $(\cP^+)^2 = \cP^+$ and
$\langle \cP^+ X, Y\rangle_2 = \langle X,\cP^+ Y\rangle_2$ for any $X,Y\in \fS_2$; see details in~\cite[Ch.~I.10]{GK2}.

With these notations, the GLM equation~\eqref{eq:pre.GLM} can be recast
as
\begin{equation}\label{eq:KGLM}
    K + \cP^+ F + \cP^+(K F) = 0
\end{equation}
or
\[
    (\cI + \cP^+_F) K = - \cP^+ F,
\]
where $\cP^+_X$ is the linear operator in $\fS_2$ defined by $\cP^+_X Y
= \cP^+ (YX)$ and $\cI$ is the identity operator in $\fS_2$. Therefore
solvability of the GLM equation is strongly connected with the
properties of the operator~$\cP^+_F$.


\begin{lemma}\label{lem:P-prop}
For every~$X\in \sB(L_2(0,1))$, the operator~$\cP^+_X$ is bounded
in~$\fS_2$. Moreover, for every $F$ from the set
\begin{equation}\label{eq:cF}
    \cF:=\{F(\bla,\bal) \mid (\bla,\bal) \in
                        \sL^s(h,r)\times \sA^s(h',r')\}
\end{equation}
the operator $\cI + \cP^+_F$ is invertible in~$\sB(\fS_2^+)$ and the inverse
$(\cI + \cP^+_F)^{-1}$ depends analytically and Lipschitz continuously
in~$\sB(\fS_2^+)$ on $F\in \cF$ in the topology of~$\fS_2$.
\end{lemma}

\begin{proof}
Boundedness of $\cP^+_X$ in $\fS_2$ is a straightforward consequence of
the inequality
\[
    \|\cP^+_X Y\|_{\fS_2} \le \|YX\|_{\fS_2} \le \|X\|_{\sB(L_2(0,1))} \|Y\|_{\fS_2},
\]
see~\cite[Ch.~3]{GK1}. Assume next that $I+X\ge\eps I$ in $L_2(0,1)$;
then, for any $Y\in\fS_2^+$,
\[
    \langle (\cI  + \cP^+_X) Y, Y\rangle_2
        = \langle Y,Y\rangle_2 + \langle YX,Y \rangle_2
        = \tr \bigl(Y(I+X)Y^*\bigr).
\]
We see that $Y(I+X)Y^* \ge \eps YY^*$ and by monotonicity of the trace
we get
\[
    \langle (\cI+\cP^+_X) Y,Y \rangle_2 \ge \eps \langle Y, Y
    \rangle_2,
\]
i.e., $\cI + \cP^+_X\ge\eps\cI$ in $\fS_2^+$.

Now, if $F=F(\bla,\bal)$ is constructed as explained in
Section~\ref{sec:GLM} from $\bla=(\la_k)_{k\in\bN} \in\sL^s(h,r)$
and $\bal=(\al_k)_{k\in\bN}\in\sA^s(h',r')$, then
\[
    I + F = \slim_{N\to\infty}
        \sum_{k=1}^N \al_k (\,\cdot\, , s_k)
        s_k
\]
with $s_k(x):=\sin\sqrt{\la_k}x$. By definition, $\al_k\ge h'$ for
all $k\in\bN$; moreover, by Theorem~\ref{thm:RB.unif} the sequence
$(s_k)_{k\in\bN}$ forms a Riesz basis of the space $L_2(0,1)$ and
its lower bound~$m>0$ can be chosen the same for all~$\bla\in
\sL^s(h,r)$. Therefore
\[
    \bigl((I+F)y,y\bigr)
        = \sum_{k=1}^\infty
        \al_k |(y,s_k)|^2
            \ge h'm\|y\|^2
\]
for every $y\in L_2(0,1)$, so that $I+F\ge h'm I$.

By the above, $\cI+\cP^+_F\ge h'm\,\cI$ in $\fS_2^+$; thus $\cI +
\cP^+_F$ is boundedly invertible in~$\sB(\fS_2^+)$ and
\[
    \|(\cI + \cP^+_F)^{-1}\| \le (h'm)^{-1}.
\]
Since $\cP^+_X$ depends linearly on~$X$, it follows that the mapping
\[
    F \mapsto (\cI + \cP^+_F)^{-1}
\]
from $\fS_2$ into~$\sB(\fS_2^+)$ is analytic and Lipschitz continuous
on the set $\cF$. The proof is complete.
\end{proof}

\begin{corollary}\label{cor:GLM}
For every $F\in\cF$, the GLM equation~\eqref{eq:KGLM} has a unique
solution
\[
    K := - (\cI + \cP^+_F)^{-1}\cP^+F \in \fS_2^+;
\]
moreover, the mapping $F\mapsto K$ from $\cF\subset\fS_2$ to $\fS_2^+$
is analytic and Lipschitz continuous.
\end{corollary}

In view of the above results and formula~\eqref{eq:pre.sigma},
the regularized potential $\si$ is determined uniquely by the function
$\phi$ of~\eqref{eq:pre.phi}. To complete the proof of
Theorem~\ref{thm:pre.norm}, we shall show that the induced mapping
$\phi \mapsto \si$ is analytic and locally Lipschitz continuous from
the space $W_2^s(0,2)$ into $W_2^s(0,1)$ for every $s\in[0,1]$. We
shall establish this for $s=0$ and~$s=1$, and then interpolate to cover
all the intermediate values $s\in(0,1)$.


\subsection{The case $s=0$.}\label{ssec:s=0}
This case is the easiest to treat, although it corresponds to
the set of Sturm--Liouville operators with the most singular
potentials---namely, with distributional potentials
in~$W_2^{-1}(0,1)$.

\begin{lemma}\label{lem:si0cont}
The function $\si$ of~\eqref{eq:pre.sigma} depends analytically and
Lipschitz continuously in~$L_2(0,1)$ on the function~$\phi\in
L_2(0,2)$ of~\eqref{eq:pre.phi} that is in the range of the mapping
of Lemma~\ref{lem:phicont} for $s=0$.
\end{lemma}

\begin{proof}
By definition,
 \(
    \si(x) = -\phi(2x) - 2 \int_0^x k(x,t) f(t,x)\, dt,
 \)
where $k$ is the kernel of the solution~$K$ of the GLM
equation~\eqref{eq:pre.GLM} and
$f(t,x)=\tfrac12[\phi(t+x)-\phi(|x-t|)]$. Thus the integral above
depends linearly on~$k$ and~$\phi$; moreover, in view
of~\eqref{eq:FS2} we get
\begin{align*}
    \int_0^1 dx &\Bigl| \int_0^x k(x,t)f(t,x)\,dt\Bigr|^2\\
        &\le \int_0^1 dx
            \int_0^x |k(x,t)|^2\,dt \int_0^x |f(t,x)|^2\,dt\\
        &\le \|\phi\|^2_{L_2(0,2)} \int_0^1\int_0^1 |k(x,t)|^2dxdt
        =  \|\phi\|^2_{L_2(0,2)} \|K\|^2_{\fS_2}.
\end{align*}
Since $K\in\fS_2^+$ depends analytically and Lipschitz continuously
on $F\in\cF$ and $F\in \fS_2$ depends linearly and continuously on
$\phi=\phi(\bla,\bal)$, the result follows.
\end{proof}

This completes the proof of Theorem~\ref{thm:pre.norm} for the case
$s=0$.


\subsection{The case $s=1$}\label{ssec:s=1}
The function $\phi$ of~\eqref{eq:pre.phi} belongs in this case
to~$W_2^1(0,2)$, and the solution $k$ to the GLM
equation~\eqref{eq:pre.GLM} must also possess some extra smoothness. We
recall that functions in $W_2^1(0,1)$ are continuous and that there
is $C>0$ such that
\[
    \max_{x\in[0,1]}|g(x)| \le C \|g\|_{W_2^1(0,1)}
\]
for every $g\in W_2^1(0,1)$. As above, we denote by $\Omega^+$ the
set~$\{(x,t)\mid 0<t<x<1\}$.

\begin{lemma}\label{lem:si1cont}
Let $s=1$ and $\phi$ and $k$ be defined as above. Then the
distributional derivative $\partial_xk$ of~$k$ belongs to
$L_2(\Omega^+)$ and the induced mapping $(\bla,\bal)\mapsto
\partial_xk$ is analytic and Lipschitz continuous from $\sL^1(h,r)\times
\sA^1(h',r')$ into $L_2(\Omega^+)$.
\end{lemma}

\begin{proof}
Let $(\bla,\bal)$ be the element of $\sL^1(h,r)\times \sA^1(h',r')$
that generates the function $\phi\in W_2^1(0,2)$. We set
\[
    \phi_n(x) := \sum_{k=1}^n
        \bigl( 2\cos \pi kx - \al_k \cos\sqrt{\la_k}x\bigr),
\]
which corresponds to taking $\al_k=\al_{k,0}:=2$ and
$\la_k=\la_{k,0}:=\pi^2k^2$ for $k>n$.  Choose $n_0$ so large that
$|\sqrt{\la_n}-\pi n|<\pi -h$ if $n>n_0$; then for
such~$n$ the sequences
\[
    (\la_k)_{k=1}^n \cup (\la_{k,0})_{k>n}
\]
and
\[
    (\al_k)_{k=1}^n \cup (\al_{k,0})_{k>n}
\]
belong to $\sL^1(h,r)$ and $\sA^1(h',r')$ respectively
and~$\phi_n\to \phi$ in~$W_2^1(0,1)$ by Lemma~\ref{lem:phicont}. We
form the kernel $f_n$ taking $\phi_n$ instead of~$\phi$
in~\eqref{eq:pre.f}; then the GLM equation with $f$ replaced by~$f_n$
has a unique solution~$k_n\in L_2(\Omega^+)$. Denoting by $F_n$ and $K_n$ the integral operators with kernels $f_n$ and $k_n$, we see that $f_n\to f$ in~$L_2\bigl((0,1)\times(0,1)\bigr)$ means that $F_n\to F$ in~$\fS_2$; hence $K_n\to K$ in~$\fS_2^+$ by Corollary~\ref{cor:GLM}, i.e., $k_n \to k$ in $L_2(\Omega^+)$.

Since the integral operator $F_n$ is
of finite rank, the solution $k_n$ can be written in an explicit form
and is easily seen to be smooth in the domain $\Omega^+$,
cf.~\cite[Sect.~12]{Fa} and~\cite[Ch.~IV.3]{GK2}. We set
$l_n:=\partial_x k_n$; then $l_n$ satisfies in~$\Omega^+$ the equation
\[
    l_n(x,t) + \int_0^x l_n(x,\xi) f_n(\xi,t)\,d\xi
        = - \tilde f_n(x,t) - k_n(x,x)f_n(x,t),
\]
where $\tilde f_n:=\partial_x f_n$. The convergence $\phi_n\to\phi$
in $W_2^1(0,2)$ implies that $f_n\to f$ in~$C(\Omega^+)$ and $\tilde
f_n(x,t) \to \tilde f(x,t):=\tfrac12[\phi'(x+t)-\phi'(x-t)]$ in
$L_2(\Omega^+)$; also,
\[
    \si_n(x):= 2k_n(x,x) - \phi_n(0)
             = -\phi_n(2x) - 2\int_0^x k_n(x,t)f_n(t,x)\,dt
\]
converge in $L_2(0,1)$ to $\si(x)$ by Lemma~\ref{lem:si0cont}. It
follows that the kernels
\[
    g_n(x,t):= \tilde f_n(x,t) + k_n(x,x)f_n(x,t)
\]
converge in $L_2(\Omega^+)$, as $n\to\infty$, to
\[
    g(x,t) := \tilde f(x,t)
        + \tfrac12 [\si(x) + \phi(0)]f(x,t),
\]
with $\tilde f := \partial_x f$. Denoting by $L_n$, $G_n$, and $G$ the
integral operators in~$\fS_2^+$ with kernels $l_n$, $g_n$, and $g$
respectively, we conclude by Lemma~\ref{lem:P-prop} that
\[
    L_n = - (\cI + \cP^+_{F_n})^{-1}G_n
        \to - (\cI + \cP^+_{F})^{-1}G =: L
\]
as $n\to\infty$ in the topology of~$\fS_2^+$. Therefore $l_n$ converge
in $L_2(\Omega^+)$ to the kernel $l$ of the operator~$L$. We conclude
that, in the sense of generalized functions, $l=\partial_x k$ and
$\partial_x k$ belongs to $L_2(\Omega^+)$ as claimed. It is easily seen
that the mapping $(\bla,\bal)\mapsto G$ from $\sL^1(h,r)\times \sA^1(h',r')$ to~$\fS_2$ is analytic and Lipschitz continuous.
We finally apply Lemma~\ref{lem:P-prop} to conclude that~$L\in\fS_2$ depends in the same manner on $F\in\cF$ and $G\in\fS_2$;
here $\cF$ is the set of~\eqref{eq:cF} corresponding to~$s=1$.
\end{proof}

To complete the proof of Theorem~\ref{thm:pre.norm} for $s=1$, it
suffices to show that the function
 $\si_1(x):= \int_0^x k(x,t)f(t,x)\,dt$
depends analytically and Lipschitz continuously in $W_2^1(0,1)$ on
$(\bla,\bal)\in\sL^1(h,r)\times \sA^1(h',r')$.  That $\si_1$ depends in
this manner in $L_2(0,1)$ was established in Subsection~\ref{ssec:s=0}.
Also,
\[
    \si_1'(x) = \tfrac14[\si(x) + \phi(0)][\phi(2x)-\phi(0)]
        + \int_0^x l(x,t)f(t,x)\,dt
        + \int_0^x k(x,t)\tilde f(t,x)\,dt,
\]
where, as in the proof of Lemma~\ref{lem:si1cont},
$l(x,t):=\partial_xk(x,t)$ and $\tilde f(t,x) :=
\partial_x f(t,x)$. Clearly, the first summand above
belongs to $L_2(0,1)$ and depends therein analytically and Lipschitz
continuously on $(\bla,\bal)$. Also, $l$ and $\tilde f$ depend in the
same manner in~$L_2(\Omega^+)$ on the spectral data (the former by
Lemma~\ref{lem:si1cont}, the latter by linearity and direct
estimates~\eqref{eq:FS2}). Thus both integrals give functions in
$L_2(0,1)$ with required dependence on $(\bla,\bal)$ (see the proof of
Lemma~\ref{lem:si0cont}), which establishes Theorem~\ref{thm:pre.norm} for $s=1$.


\subsection{The case $s\in (0,1)$}\label{ssec:s-general}
The general case will be covered by the nonlinear interpolation
theorem due to Tartar~\cite{Ta}, which implies the following result.

\begin{proposition}\label{pro:Tartar}
Assume that $(X_0, X_1)$ and $(Y_0,Y_1)$ are pairs of Banach spaces
with continuous embeddings $X_1\hookrightarrow X_0$ and
$Y_1\hookrightarrow Y_0$. Let also $\Phi\,:\, X_0 \to Y_0$ be a
nonlinear mapping that is Lipschitz continuous on the
balls~$B_{X_0}(r):=\{x\in X_0 \mid \|x\|_{X_0} \le r\}$ for every
$r>0$. Assume further that $\Phi X_1 \subset Y_1$ and that $\Phi$ is
Lipschitz continuous on every ball~$B_{X_1}(r)$ of $X_1$ as a mapping
from $X_1$ into $Y_1$. Construct the interpolation spaces
$X_s:=[X_1,X_0]_s$ and $Y_s:=[Y_1,Y_0]_s$, $s\in(0,1)$, by the complex
interpolation method; then $\Phi$ acts boundedly from $X_s$ to $Y_s$
for every $s\in(0,1)$ and, moreover, its restriction to the
ball~$B_{X_s}(r)$ of $X_s$ is Lipschitz continuous for every $r>0$. In
other words, for every $r>0$ there is $C=C(r,s)$ such that
\[
    \| \Phi(x_1)-\Phi(x_2) \|_{Y_s} \le C \|x_1-x_2\|_{X_s}
\]
whenever $x_1$ and $x_2$ belong to~$B_{X_s}(r)$.
\end{proposition}

By definition, the spaces $\ell_2^s$, $\hat\ell_2^s$ and $W_2^s(0,1)$,
as well as their ``even'' and ``odd'' subspaces, form the interpolation
space scales. Therefore the above proposition, in view of the results
of Subsections~\ref{ssec:s=0} and \ref{ssec:s=1}, implies that for
every $h\in(0,\pi)$, $h'\in(0,2)$, positive $r$ and $r'$, and
$s\in(0,1)$ the mapping $(\bla,\bal)\mapsto\sigma$ is Lipschitz
continuous from~$\sL^s(h,r)\times\sA^s(h',r')$ to $W^s_2(0,1)$.
Analyticity for all $s\in[0,1]$ again follows from that for $s=0$ and
$s=1$ and interpolation theorem for linear operators~\cite{BL,LM} applied to the
Fr\'echet derivative of this mapping.

\begin{remark}
For $s<\tfrac12$, the arguments of the paper~\cite{HMscale} combined with
the results of the previous subsection provide a direct proof of
analytic and Lipschitz continuous dependence of $\si\in W_2^s(0,1)$
on the spectral data $(\bla,\bal) \in \sL^s(h,r)\times
\sA^s(h',r')$.
\end{remark}

\begin{remark}
It should be clear how one can iterate the considerations of
Subsection~\ref{ssec:s=1} to all natural values of~$s$ and then
interpolate as in Subsection~\ref{ssec:s-general} to get all
positive~$s$, cf.~\cite[Sect.~3.4]{Ma} and~\cite{SSstab}.
\end{remark}


\section{Reconstruction from two spectra}\label{sec:two}

In this section, we complete the proof of Theorem~\ref{thm:pre.main} by
establishing uniform continuity of the norming constants on the two
spectra. More exactly, given an
element~$\bigl((\la_n),(\mu_n)\bigr)\in\sN^s(h,r)$, we define
numbers~$\al_n$ by the relation
\begin{equation}\label{eq:two.al}
    \al_n = \frac{2\sqrt{\la_n}}{\dot{S}(\sqrt{\la_n})C(\sqrt{\la_n})},
\end{equation}
where the entire functions $S$ and $C$ are given by the infinite
products~\eqref{eq:pre.PhiPsi}, and prove the following theorem.

\begin{theorem}\label{thm:two.norm}
For every $s\in[0,1]$, $h\in(0,\pi/2)$ and $r>0$ the mapping
\[
   \sN^s(h,r) \ni (\bla,\bmu) \mapsto \bal \in \sA^s
\]
is analytic and Lipschitz continuous; moreover, there are $h'>0$ and
$r'>0$ such that the range of this mapping belongs
to~$\sA^s(h',r')$.
\end{theorem}

It suffices to show that there are $h''>0$ and $r''>0$ such that the
numbers $\tilde\beta_n$ defined via the relation
\[
    \frac2{\al_n} = 1 + \tilde\beta_n
\]
satisfy the inequality $1+\tilde\beta_n\ge h''$ and form an $\ell^s_2$-sequence
$\tilde \bbe := (\tilde\beta_n)$ depending analytically and
Lipschitz continuously on~$(\bla,\bmu)\in\sN^s(h,r)$ and having norm not greater than~$r''$. Indeed, then
$\al_n \ge h':= 2/(1+r'')$ and $\beta_{2n}:=\al_n-2=-2\tilde\beta_n/(1+\tilde\beta_n)$. We observe that, for a bounded sequence $(d_n)$, the mapping $(x_n)\mapsto(d_nx_n)$ is a bounded linear operator in~$\ell_2^s$ of norm~$d:=\sup_n|d_n|$. Since $\bigl|-2/(1+\tilde\beta_n)\bigr|\le 2/h''$, we conclude that the sequence~$(\beta_{2n})$ belongs to $\ell_2^s$ and has the norm at most~$2r''/h''$; the sequence $(\beta_n)$ with $\beta_{2n-1}=0$ belongs then to~$\ell_{2,\mathrm{even}}^s$ and is of norm at most~$r':=2^{s/2+1}r''/h''$.

To justify analyticity and Lipschitz continuity of $\bal$ in~$\sA^s$, we exploit the fact that $\ell_2^s$ is a Banach algebra under the point-wise multiplication $(x_n)\cdot (y_n) = (x_ny_n)$. We denote by~$A$ the unital extension of $\ell_2^s$; elements of~$A$ have the form $a\1 + \bx$, where $a\in\bC$, $\1$ is the sequence with all its elements equal to~$1$, and $\bx=(x_n)\in \ell_2^s$, and the norm in~$A$ is given by $\|a\1 +\bx\|_A := |a| +\|\bx\|_s$. An element $a\1 +\bx$ is invertible in~$A$ if and only if $a\ne0$ and $a+x_n\ne0$ for all~$n\in\bN$; in this case the inverse is equal to $\tfrac1a\1 +\by$ with $\by=(y_n)$ and $y_n:=-x_n/[a(a+x_n]$. Since the sequence $\bigl(1/[a(a+x_n)]\bigr)$ is bounded, the above reasoning shows that indeed $\by$ belongs to $\ell_2^s$.

We now see that
\[
    \tfrac12 \bal = (\1 + \tilde \bbe)^{-1};
\]
since taking an inverse element is an analytic mapping in a unital Banach algebra, $\bal$ depends analytically on~$\tilde\bbe$. Lipschitz continuity of~$\bal$ follows from the fact that, for the set of~$\tilde\bbe$ considered, $\1 + \tilde \bbe$ have uniformly bounded inverses in~$A$ (of norm not greater than $1+r''/h''$).

To the rest of this section, we shall use the notations
$\omega_{2n}:=\sqrt{\la_n}$ and $\omega_{2n-1}:=\sqrt{\mu_n}$ and
$\rho_n:=\omega_n-\pi n/2$. In view of~\eqref{eq:two.al} we have
\[
    1 + \tilde\beta_n
        = \dot{S}(\omega_{2n})\,
            \frac{C(\omega_{2n})}{\omega_{2n}}
        =:\bigl(1+a_n\bigr)\bigl(1+b_n\bigr),
\]
where we set
\begin{equation}\label{eq:two.cndn}
    a_n := (-1)^n\dot{S}(\omega_{2n}) -1,
    \qquad
    b_n := (-1)^n\frac{C(\omega_{2n})}{\omega_{2n}} -1.
\end{equation}
Thus we need to prove that both $\dot{S}(\omega_{2n})$ and
${C(\omega_{2n})}/{\omega_{2n}}$ are uniformly bounded away from zero,
that the sequences $(a_n)$ and $(b_n)$ are the sequences of the even
cosine Fourier coefficients of some functions $h_1$ and $h_2$ from
$W_{2,\mathrm{even}}^s(0,1)$ of zero mean, and that the mappings
\begin{equation}\label{eq:two.cn}
    \sN^s(h,r) \ni (\bla,\bmu) \mapsto h_1 \in W_{2,\mathrm{even}}^s(0,1)
\end{equation}
and
\begin{equation}\label{eq:two.dn}
    \sN^s(h,r) \ni (\bla,\bmu) \mapsto h_2 \in W_{2,\mathrm{even}}^s(0,1)
\end{equation}
are analytic and Lipschitz continuous. We do this in the two
subsections that follow.

\subsection{Analyticity and continuity}

\begin{lemma}\label{lem:two.cndn}
The mappings of~\eqref{eq:two.cn} and \eqref{eq:two.dn} are analytic
and Lipschitz continuous.
\end{lemma}

\begin{proof}
We observe that in view of~\eqref{eq:pre.y} and \eqref{eq:pre.y'} the
functions $\dot{S}(z)$ and $C(z)/z$ have the representation
\[
    \cos z + \int_0^1 g(t) \cos z(1-2t)\,dt,
\]
with $g(t) = (1-2t) k(1,t)\in L_2(0,1)$ for $\dot S(z)$ and $g(t)=
k_1(1,t)\in L_2(0,1)$ for $C(z)/z$. Therefore both $a_n$ and $b_n$ can
be written as
\[
    [(-1)^{n}\cos\omega_{2n} -1] +
    (-1)^n\int_0^1 g(t) \cos\omega_{2n}(1-2t)\,dt =: d_n + e_n,
\]
with respective $g\in L_2(0,1)$. Clearly, we may (and shall) take
the even part $g_{\mathrm{even}}$ of~$g$ instead of~$g$ in the above
integral.

Recalling that $\omega_{2n}= \pi n + s_{2n}(f)$ for a (unique) function $f\in
W_{2,\mathrm{odd}}^s(0,1)$ and observing that $s_{2n}(f)=\ri \hat
f(n)$, with $\hat f(n)$ being the $n$-th Fourier coefficient of a function~$f$ (see Appendix~\ref{sec:sobolev}), we can write the $d_n$ as
 \[
    d_n = \cos s_{2n}(f)-1
        = \cosh \hat f(n) - 1
        = \sum_{k=1}^\infty \frac{\hat f(n)^{2k}}{(2k)!}.
 \]
It follows that the $d_n$ is the~$n$-th Fourier coefficient of the
function~$\tilde f$ given by
\[
    \tilde f:= \sum_{k=1}^\infty \frac{f^{<2k>}}{(2k)!},
\]
where $f^{<k>}$ is the $k$-fold convolution of~$f$ with itself. The
function~$\tilde f$ has zero mean and is even with respect
to~$x=\tfrac12$ and therefore~$d_n = c_{2n}(\tilde f)$. By the results
of Appendix~\ref{sec:sobolev} the mapping $f\mapsto\tilde f$ is
analytic from~$W_{2,\mathrm{odd}}^s(0,1)$
to~$W_{2,\mathrm{even}}^s(0,1)$ and is Lipschitz continuous on bounded
subsets of~$f$.

Recalling Lemma~\ref{lem:C.Psi} and the remark with formula~\eqref{eq:C.2} following it, we see that the~$e_n$ give the $2n$-th cosine Fourier coefficient of the function
$h:=\tfrac12\Psi(\ri f, g)$ in $W_{2,\mathrm{even}}^s(0,1)$ of zero mean, which
depends analytically and boundedly Lipschitz continuously on~$f$
and~$g$.

It remains to apply Lemma~\ref{lem:two.kk1} below to conclude that
the even parts of the functions $(1-2t)k(1,t)$ and $k_1(1,t)$, which
we have taken as the function~$g$ above, depend analytically and
Lipschitz continuously in $W_2^s(0,1)$ on the spectral data
in~$\sN^s(h,r)$. This completes the proof of the lemma.
\end{proof}

\begin{lemma}\label{lem:two.kk1}
For every $s\in [0,1]$, the mappings
\[
    \sN^s(h,r) \ni (\bla,\bmu)  \mapsto k_{\mathrm{odd}}(1,\cdot) \in W_2^s(0,1)
\]
and
\[
    \sN^s(h,r) \ni (\bla,\bmu)  \mapsto k_{1,\mathrm{even}}(1,\cdot) \in W_2^s(0,1)
\]
are analytic and Lipschitz continuous. Here $k_{\mathrm{odd}}(1,\cdot)$
and~$k_{1,\mathrm{even}}(1,\cdot)$ are respectively odd and even parts
(with respect to~$x=\tfrac12$) of the functions~$k(1,\cdot)$ and
$k_1(1,\cdot)$.
\end{lemma}

\begin{proof}
Since both mappings are treated similarly, we only consider in detail
the first one. By definition of the function $y(\cdot,z)$ we have
$y(1,\omega_{2n})=0$, and thus the numbers $\omega_{2n} = \pi n +
s_{2n}(f)$ along with $0$ and $-\omega_{2n}$ are zeros of the entire
function
\[
    y(1,z)= \sin z + \int_0^1 k(1,t) \sin z (1-2t)\,dt.
\]
This function does not depend on the even part
$k_{\mathrm{even}}(1,\cdot)$ of the function~$k(1,\cdot)$; in fact, it
can be written in the form
\begin{equation}\label{eq:two.g}
    \sin z + \int_0^1 g(t) \re^{\ri z(1-2t)}\,dt
\end{equation}
with $g(t):= -\ri k_{\mathrm{odd}}(1,t)$. Since the set of
functions~$\{\sin \omega_{2n}(1-2t)\}_{n\in\bN}$ is complete in
$L_{2,\mathrm{odd}}(0,1)$ (cf.~the results of Appendix~\ref{sec:RB}),
one can show that~$k_{\mathrm{odd}}(1,\cdot)$ is uniquely determined by
the zeros $\omega_{2n} = \pi n + s_{2n}(f) = \pi n + i \hat f(n)$. Some
important properties of the induced mapping $f\mapsto g$ can be derived
from the paper~\cite{HMzero}.

Indeed, the results of~\cite{HMzero} imply that for every $f\in
W_2^s(0,1)$ there exists a unique function~$g\in W_2^s(0,1)$ such that
all zeros (counting multiplicities) of the entire
function~\eqref{eq:two.g} are given  by the numbers $\pi n + \hat
f(n)$, $n\in\bZ$. Such pairs of $f$ and $g$ in fact satisfy the
relation
\begin{equation}\label{eq:2.H}
    H(f,g) := s(f) + g +
        \sum_{k=1}^\infty \frac{(M^k g)\ast f^{<k>}}{k!} =0;
\end{equation}
here
\[
    s(f):= \sum_{k=0}^\infty\frac{(-1)^k f^{<2k+1>}}{(2k+1)!},
\]
$f^{<k>}$ is the $k$-fold convolution of~$f$ with itself, and $M$ is
the operator of multiplication by~$\ri (1-2x)$. The function $H$ is
analytic from $W_2^s(0,1)\times W_2^s(0,1)$ into $W_2^s(0,1)$, and its
partial derivatives $\partial_f H(f,g)$ and $\partial_g H(f,g)$ are
given by
\begin{align}\label{eq:5.partialf}
    \partial_f H(f,g)(h_1) &=
         \Bigl( c(f) + \sum_{k=1}^\infty \frac{(M^k g)\ast f^{<k-1>}}{(k-1)!}
         \Bigr) \ast h_1,\\
    \partial_g H(f,g)(h_2) &=
        h_2 + \sum_{k=1}^\infty \frac{(M^k h_2)\ast f^{<k>}}{k!}
        \label{eq:5.partialg}
\end{align}
with
\[
    c(f) := \sum_{k=0}^\infty \frac{(-1)^k f^{<2k>}}{(2k)!}.
\]

We assume now that $f\in W_{2,\mathrm{odd}}^s(0,1)$ is such that the
corresponding sequence $\bla=(\la_n)_{n\in\bN}$ with $\la_n :=
\omega^2_{2n}=(\pi n + s_{2n}(f))^2$ belongs to $\sL^s(h,r)$. Set
$S_{\bla}$ to be the function of~\eqref{eq:pre.PhiPsi}; then $S_{\bla}$ can also be represented as~\eqref{eq:two.g}. Direct
calculations show that the $n$-th Fourier coefficient of the function
of~\eqref{eq:5.partialf} is equal to
\[
    (-1)^n \hat h_1(n) \Bigl[ \cos \omega_{2n}
            + \int_0^1 \ri (1-2t)g(t) \re^{\ri \omega_{2n}(1-2t)}\,dt \Bigr]
            = (-1)^n \hat h_1(n) \dot{S}_{\bla}(\omega_{2n}).
\]
By Lemma~\ref{lem:two.PhiPsi} below there are numbers $K_1$ and $K_2$ such
that
\[
    K_1 \le |\dot{S}_{\bla}(\omega_{2n})| \le K_2
\]
for all $\bla\in \sL^s(h,r)$ and all $n\in\bN$. The results of
Appendix~\ref{sec:sobolev} imply that the operator $\partial_f H(f,g)$
is bounded in every space $W_2^s(0,1)$ and its norm is at most~$K_2$.

Similarly, the $n$-th Fourier coefficient of the function
of~\eqref{eq:5.partialg} is equal to
\[
    (-1)^n \int_0^1 h_2(t) \re^{\ri \omega_{2n}(1-2t)}\,dt.
\]
By Theorem~\ref{thm:RB.unif} there exist positive $M$ and $m$ such
that, for all $\bla\in \sL^s(h,r)$, the sequences
 $(\re^{\ri\omega_{2n}(1-2x)})_{n\in\bZ}$
form Riesz bases of $L_2(0,1)$ of upper bound~$M$ and lower
bound~$m$, see Appendix~\ref{sec:RB}. Therefore the
operator~$H_g:=\partial_g H(f,g)$,
\[
   H_g :\,h_2 \mapsto
        \sum_{n\in\bZ} (-1)^n(h_2, \re^{\ri\omega_{2n}(2x-1)})
            \,\re^{2\pi n\ri x},
\]
is bounded and boundedly invertible in~$L_2(0,1)$, with $\|H_g\|\le
M^{1/2}$ and $\|H_g^{-1}\|\le m^{-1/2}$. If $h_2 \in W_2^1(0,1)$, then
we integrate by parts to get
\[
    c_n:=(h_2, \re^{\ri\omega_{2n}(2x-1)})
        = \frac1{2\ri\omega_{2n}}
            [h_2(0)\re^{\ri\omega_{2n}} - h_2(1)\re^{-\ri\omega_{2n}}]
            + \frac1{2\ri\omega_{2n}}
              (h_2',\re^{\ri\omega_{2n}(2x-1)}).
\]
It is clear that the sequence $(c_n)_{n\in\bZ}$ forms an element of
$\tilde \ell_2^1(\bZ)$, see Section~\ref{ssec:B.FT}. Thus the
operator $H_g$ acts boundedly and boundedly invertible
in~$W_2^1(0,1)$, and it remains to use the interpolation theorem to
derive the same properties of~$H_g$ in $W_2^s(0,1)$ for all
$s\in[0,1]$.

We now use the implicit mapping theorem to conclude that the mapping
$f\mapsto g$ is analytic in $W_2^s(0,1)$. Recalling the isomorphism of
the space $\sL^s$ of sequences~$\bla=(\la_n)$ of the Dirichlet
eigenvalues of the Sturm--Liouville operators $T(\si)$ with $\si\in
W_2^s(0,1)$ and the subspace~$W_{2,\mathrm{odd}}^s(0,1)$ explained in
Section~\ref{sec:pre}, we conclude that the mapping
\[
    \sL^s(h,r) \ni \bla \mapsto k_{\mathrm{odd}}(1,\cdot) \in
        W_{2,\mathrm{odd}}^s(0,1)
\]
is analytic. The uniform bounds on the inverses of the partial
derivatives $\partial_f H(f,g)$ and $\partial_g H(f,g)$ established
above imply that this mapping is Lipschitz continuous, and the proof for the first mapping is
complete.

The second mapping of the lemma is treated analogously using the
relations
\[
   \cos \omega_{2n-1}
        + \int_0^1 k_1(1,t)\cos\omega_{2n-1} (1-2t)\,dt = 0
\]
and the Riesz basis properties of the
system~$(\cos\omega_{2n-1}t)_{n\in\bN}$,
cf.~Remark~\ref{rem:RB-cos-mu}.
\end{proof}


\subsection{Uniform positivity of $\al_n$} Since $\sN^s(h,r) \subset
\sN(h,r)$ if $s\ge0$, it only suffices to consider the case $s=0$.
In view of formula~\eqref{eq:two.al}, uniform positivity of $\al_n$
immediately follows from the lemma below.

\begin{lemma}\label{lem:two.PhiPsi}
For every $h\in (0,\pi/2)$ and $r>0$ we have
\[
    \sup_{(\bla,\bmu)} \sup_{n\in\bN}\,
            \log|\dot{S}(\omega_{2n})|<\infty,
    \qquad
    \sup_{(\bla,\bmu)} \sup_{n\in\bN}\,
            \log\frac{|C(\omega_{2n})|}{\omega_{2n}} <\infty,
\]
where $S$ and $C$ are constructed via~\eqref{eq:pre.PhiPsi} from
sequences $\bla:=(\omega^2_{2k})$ and $\bmu:=(\omega^2_{2k-1})$, and the
suprema are taken over $(\bla,\bmu)\in \sN(h,r)$.
\end{lemma}

\begin{proof}
By~\eqref{eq:pre.PhiPsi}, we have
\[
    \dot{S}(\sqrt{\la_n}) = -\frac{2\la_n}{\pi^2 n^2}
    \prod_{k\in\bN,\ k\ne n}\frac{\la_k - \la_n}{\pi^2k^2}.
\]
Dividing both sides by
\[
    \cos \pi n = \frac{d \sin z}{dz}\Bigr|_{z=\pi n}
                = -2 \prod_{k\in\bN,\ k\ne n}
                    \frac{\pi^2k^2-\pi^2n^2}{\pi^2k^2},
\]
we conclude that%
\begin{footnote}
{In what follows, all summations and multiplications over the index set
$\bZ$ will be taken in the principal value sense and the
symbol~$\mathrm{V.p.}$ will be omitted.}
\end{footnote}
\[
    |\dot{S}(\sqrt{\la_n})| = \frac{\la_n}{\pi^2n^2}
            \prod_{k\ne n} \frac{\la_n - \la_k}{\pi^2n^2-\pi^2k^2}
            = \prod_{k\in\bZ,\ k\ne n}
                      \frac{\omega_{2k} - \omega_{2n}}{\pi(k-n)},
\]
where we set $\omega_{-k}:=-\omega_k$ for $k\in\bN$ and
$\omega_0:=0$. Set also (recall that $\rho_k:=\omega_k-\pi k$)
\[
    a_{k,n}:= \frac{\omega_{2k} - \omega_{2n}}{\pi(k-n)} -1
            = \frac{\rho_{2k}-\rho_{2n}}{\pi (k-n)}
\]
if $k\ne n$ and $a_{n,n}:=0$; then
    \(
    |\dot{S}(\omega_{2n})| = \prod_{k\in\bZ}
        (1+ a_{k,n}).
    \)
Since the sequence $(\omega_n)$ is $h$-separated for every
$(\bla,\bmu)\in\sN(h,r)$, we have
 \(
    1+a_{k,n} \ge {2h}/\pi
 \)
for all integer $k$ and~$n$. Therefore, with
\[
    K:= \max_{x\ge -1 + 2h/\pi} \Bigl|\frac{\log(1+x)-x}{x^2}\Bigr|
    < \infty,
\]
we get the estimate
\begin{equation}\label{eq:two.log}
    \Bigl|\log \prod_{k\in\bZ}(1+a_{k,n})\Bigr|
            \le \Bigl|\sum_{k\in\bZ} a_{k,n}\Bigr|
            + K \sum_{k\in\bZ} a^2_{k,n},
\end{equation}
provided the two series converge.

Clearly,
\[
    \sum_{k\ne n} \frac{1}{k-n}  =0,
\]
and thus
\[
    \Bigl|\sum_{k\in\bZ} a_{k,n}\Bigr|
        = \Bigl|\frac1\pi\sum_{k\ne n} \frac{\rho_{2k}}{k-n}\Bigr|
        \le \frac{\sqrt 2r}{\sqrt3}
\]
by the Cauchy--Bunyakovski--Schwarz inequality (recall that
$\sum_{k\in\bZ}\rho^2_{2k} \le 2r^2$ by the definition of the set
$\sN(h,r)$ and $\sum_{k\ne n}(k-n)^{-2}=\pi^2/3$). Next, the
inequality
\[
    a_{k,n}^2 \le  \frac{2 \rho_{2k}^2}{(k-n)^2}
                   + \frac{2 \rho_{2n}^2}{(k-n)^2}
\]
for $k\ne n$ yields
\[
    \sum_{k\in\bZ} a_{k,n}^2
        \le 4r^2 \sum_{k\ne n}\frac{1}{(n-k)^2}
        = \frac{4\pi^2r^2}3.
\]
It follows from~\eqref{eq:two.log} that
\[
    \Bigl|\log \prod_{k\in\bZ}(1+a_{k,n})\Bigr|
        \le (\sqrt6 r + 4K \pi^2r^2)/3,
\]
where the constant~$K$ only depends on $h$.

Similarly, we find that
\[
    \Bigl|\frac{C(\sqrt{\la_n})}{\sqrt{\la_n}}\Bigr|
        = \Bigl|\prod_{k=1}^\infty
        \frac{\mu_k-\la_n}{\pi^2(k-\tfrac12)^2}\Bigr|
        = \prod_{k\in\bZ}
        \frac{\omega_{2k-1}-\omega_{2n}}{\pi(k-\tfrac12) - \pi n}
\]
and then mimic the above reasoning to establish the other uniform bound.
The lemma is proved.
\end{proof}

As explained at the beginning of this Section, the above statements complete the proofs of Theorems~\ref{thm:two.norm} and \ref{thm:pre.main}.

\medskip
\textbf{Acknowledgements.} {The author thanks A.~A.~Shkalikov for
suggesting the problem and M.~Marletta, Ya.~V.~Mykytyuk and R.~Weikard for stimulating discussions. The research was partially supported by the Alexander von Humboldt Foundation and was partially carried out during the visit to the
Institute for Applied Mathematics of Bonn University, whose hospitality
is warmly acknowledged.}

\appendix



\section{Riesz bases of sines and cosines}\label{sec:RB}

We recall (see, e.g., \cite[Ch.~6]{GK1} and \cite[Ch.~4]{Young})
that a sequence~$(f_n)_{n\in\bN}$ in a separable Hilbert space~$H$
is called a Riesz basis of~$H$ if it is a homeomorphic image of an orthonormal basis of~$H$. Then there are $M>0$ (the upper bound)
and $m>0$ (the lower bound) such that, for every $f\in H$, we have
\[
    m\|f\|^2 \le \sum \bigl|(f,f_n)\bigr|^2 \le M \|f\|^2.
\]
Riesz bases of $L_2(0,1)$ that are composed of exponential
functions, or sines, or cosines, have been extensively studied in
the literature starting from the early 1930-ies, see the books by
Paley and Wiener~\cite{PW} and Avdonin and Ivanov~\cite{AI} for
particulars and historical comments. For instance, the famous Kadets
$\tfrac14$-theorem~\cite{Ka} implies that for every $L<\tfrac14$ there exist
positive constants~$m$ and $M$ such that as long as a sequence
$(\omega_n)_{n\in\bZ}$ of real numbers satisfies the condition
\begin{equation}\label{eq:RB.Kad}
    \sup_{n\in\bZ} |\omega_n-\pi n|< \pi L,
\end{equation}
then the sequence $(\re^{i\omega_n x})_{n\in\bZ}$ of exponentials
forms a Riesz basis of~$L_2(-1,1)$ of upper bound~$M$ and lower
bound~$m$. Analogous results for families of sines and cosines were
established in~\cite{HV}.

In this paper, we need generalizations of these results to sequences
that may not satisfy condition~\eqref{eq:RB.Kad}. Recall that
$\sL^0(h,r)$, with $h\in(0,\pi )$ and $r>0$, stands for the set of
all strictly increasing sequences $\bla=(\omega^2_n)_{n\in\bN}$ of
positive numbers satisfying the conditions $\omega_1>h$,
$\omega_{n+1}-\omega_n\ge h$, $n\in\bN$, and
 \(
    \sum |\omega_n - \pi n|^2 \le r^2.
 \)
For every $\bla\in \sL^0(h,r)$, we denote by $\sS_{\bla}$ and
$\sC_{\bla}$ the sequences of functions $(\sin\omega_nx)_{n\in\bN}$ and
$(\cos\omega_nx)_{n\in\bZ_+}$ respectively, with $\omega_0:=0$. We also
set $\omega_{-n}:= -\omega_n$ and denote by~$\sE_{\bla}$ the sequence
of functions~$(\re^{\ri \omega_n(1-2x)})_{n\in\bZ}$.  The following
statement can be derived from the results of~\cite{Hriesz}:

\begin{theorem}\label{thm:RB.unif}
For every $h\in(0,\pi)$ and $r>0$ there exist positive numbers $M$ and
$m$ such that for every $\bla\in \sL^0(h,r)$ the sequences
$\sS_{\bla}$, $\sC_{\bla}$, and $\sE_{\bla}$ are Riesz bases of
$L_2(0,1)$ of upper bound~$M$ and lower bound~$m$.
\end{theorem}

\begin{remark}\label{rem:RB-cos-mu}
For $h\in(0,\pi)$ and $r>0$, we denote by $\sM^0(h,r)$ the set of
increasing sequences~$\bmu:=(\mu_n)_{n=1}^\infty$ with the following
properties:
\begin{itemize}
\item [(M1)] $\mu_1\ge1$ and, for all $n\in\bN$, $\sqrt{\mu_{n+1}} - \sqrt{\mu_n}\ge h$;
\item [(M2)] the numbers~$\rho_{n}:=\sqrt{\mu_n}- \pi (n-\tfrac12)$
form a sequence in $\ell_2$ of norm at most~$r$.
\end{itemize}
Then an analogue of the above theorem holds true for the family of
sequences~$\sC_{\bmu}=(\cos\sqrt{\mu_n}x)_{n\in\bN}$, with $\bmu$
running through the set~$\sM^0(h,r)$; see~\cite{Hriesz}.
\end{remark}


\section{Sobolev spaces $W^s_2(0,1)$ and some of their properties}\label{sec:sobolev}

We recall here some facts about the Sobolev spaces $W^s_2(0,1)$ and
Fourier coefficients of functions from these spaces. For details, we
refer the reader to~\cite[Ch.~1]{LM}.

\subsection{The definition}
By definition, the space $W^0_2(0,1)$ coincides with $L_2(0,1)$ and
the norm~$\|\cdot\|_0$ in $W^0_2(0,1)$ is just the $L_2(0,1)$-norm.
For a natural~$l$, the Sobolev space $W^l_2(0,1)$ consists of all
functions $f$ in $L_2(0,1)$, whose distributional derivatives
$f^{(k)}$ for $k=1,\dots,l$ also fall into~$L_2(0,1)$. Being endowed
with the norm
\begin{equation}\label{eq:B.sob-norm}
    \|f\|_l := \Bigl(\sum_{k=0}^l \|f^{(k)}\|_0^2\Bigr)^{1/2},
\end{equation}
the space $W_2^l(0,1)$ becomes a Hilbert space.

The intermediate spaces $W^s_2(0,1)$ for arbitrary positive $s$ can
be constructed by interpolation~\cite[Ch.~1.2.1]{LM}. We shall need
such spaces only for $s\in[0,2]$ and thus interpolate between
$W^2_2(0,1)$ and $W^0_2(0,1)$ to get them, i.e.,
\[
    W^{2t}_2(0,1) :=
        [W^2_2(0,1),W^0_2(0,1)]_{1-t}, \quad t\in(0,1).
\]
The induced norms $\|\cdot\|_s$ (for $s=1$ the
norm~\eqref{eq:B.sob-norm} is equivalent to that defined by
interpolation) are nondecreasing with $s\in[0,2]$, i.e., if $s<r$
and $f\in W^r_2(0,1)$, then $\|f\|_s\le\|f\|_r$. Since by
construction the spaces $W_2^s(0,1)$ form an interpolation scale,
the general interpolation theorem~\cite[Theorem~1.5.1]{LM} implies
the following interpolation property for operators in these spaces.

\begin{proposition}\label{pro:interp}
Assume that an operator $T$ acts boundedly in $W^s_2(0,1)$ and
$W^r_2(0,1)$, $s<r$. Then $T$ is a bounded operator in
$W^{ts+(1-t)r}_2(0,1)$ for every $t\in[0,1]$; moreover,
$\|T\|_{ts+(1-t)r}\le \|T\|_s^t \|T\|_r^{1-t}$.
\end{proposition}

Proposition~\ref{pro:interp} yields boundedness in every
$W^s_2(0,1)$, $s\in [0,2]$, of the reflection operator $R$ given by
$Rf(x)=f(1-x)$ and the operator~$M$ of multiplication by $\ri x$,
$Mf(x):=\ri xf(x)$.


\subsection{Fourier transform}\label{ssec:B.FT}
For an arbitrary $f\in L_2(0,1)$ we denote by $\hat f:\,\bZ \to \bC$
its discrete Fourier transform, viz.
\[
    \hat f(n):= \int_0^1 f(x) \re^{-2\pi i nx}\,dx.
\]
The Fourier transform is a unitary mapping between~$L_2(0,1)$ and
$\ell_2(\bZ)$. We set
\[
    \tilde\ell_2^s(\bZ) := \begin{cases}
        \ell_2^s(\bZ)    & \text{if} \qquad s<\tfrac12,\\
        \ell_2^s(\bZ) \dotplus \ls\{{\mathbf e}^{(1)}\}
                    & \text{if} \qquad \tfrac12\le s<\tfrac32,\\
        \ell_2^s(\bZ) \dotplus \ls\{{\mathbf e}^{(1)},\mathbf{e}^{(2)}\},
                    & \text{if} \qquad \tfrac32\le s\le 2,
        \end{cases}
\]
where $\mathbf{e}^{(j)}$, $j=1,2$, denotes the sequence
$(e^{(j)}_n)_{n\in\bZ}$ with $e^{(j)}_n := n^{-j}$ for $n\ne 0$ and
$e^{(j)}_0:=0$. The norm $\|\cdot\|_s$ in $\hat \ell_2^s(\bZ)$ is defined as
follows: given an element $\bx:=(x_n)$ of~$\ell_2^s(\bZ)$ and
complex numbers $a_1$ and $a_2$ (with $a_1=a_2=0$ if $s<\tfrac12$ and $a_2=0$ if $\tfrac12\le s<\tfrac32$), we set
\[
    \|\bx + a_1{\mathbf e}^{(1)} + a_2{\mathbf e}^{(2)}\|^2_s:=
        \sum_{n\in\bZ}(1+n^2)^s |x_n|^2 + |a_1|^2 + |a_2|^2.
\]
It is straightforward to verify that the Fourier transform of every
function in $W_2^2(0,1)$ forms an element of $\tilde \ell_2^2(\bZ)$ and that
$\|\hat f\|_2$ introduces a norm on $W_2^2(0,1)$ that is equivalent to
that of~\eqref{eq:B.sob-norm}. Since the spaces $\tilde \ell_2^s(\bZ)$,
$s\in[0,2]$, form an interpolation scale, the interpolation theorem
implies that, for every $s\in[0,2]$, the Fourier transform is a
homeomorphism between the spaces $W_2^s(0,1)$ and $\tilde \ell_2^s(\bZ)$.

\subsection{Convolution}
As usual, $\ast$ denotes the convolution operation on~$(0,1)$, viz.
\[
    (f\ast g)(x) := \int_0^1 f(x-t) g(t)\,dt;
\]
here we extend a function $f$ onto the interval~$(-1,0)$ by
periodicity, i.e., by setting $f(x):=f(x+1)$ for $x\in (-1,0)$. It
is well known that convolution accumulates smoothness; the precise
meaning of this statement is as follows.

\begin{proposition}\label{pro:B.conv}
Assume that $s,t\in[0,1]$ and that $f\in W^s_2(0,1)$ and $g\in
W^t_2(0,1)$ are arbitrary. Then the function $h:= f\ast g$ belongs
to $W^{s+t}_2(0,1)$ and, moreover, there exists $C>0$ independent of
$f$ and $g$ such that $\|h\|_{s+t}\le C \|f\|_s\|g\|_t$.
\end{proposition}

Proof of this proposition is based on interpolation between the
extreme cases $s,t=0,1$, which are handled with directly using the
representation
\[
    (f \ast g) (x) = \int_0^x f(x-t) g(t)\,dt
        + \int_x^1 f(x-t+1)g(t)\,dt.
\]
We also recall the relation~$\widehat{f\ast g}(n) = \hat f(n) \hat
g(n)$.


\section{Some auxiliary results}\label{sec:aux}

\begin{lemma}\label{lem:C.Phi}
For $f$ and $g$ in~$L_2(0,1)$, set
\[
    \Phi(f,g):= \mathrm{V.p.}\sum_{n\in\bZ}
        \hat g(n) [\exp\{\hat f(n)\ri x\}-1] \re^{2\pi\ri nx};
\]
then for every $s,t\in[0,1]$ the mapping
\[
    \Phi\,:\, W_2^s(0,1)\times W_2^t(0,1) \to W_2^{s+t}(0,1)
\]
is analytic and Lipschitz continuous on bounded subsets.
\end{lemma}

\begin{proof}
If $f$ and $g$ are in $L_2(0,1)$, then the series for $\Phi$ converges
absolutely, and thus $\Phi(f,g)$ is a continuous function. Take now
$f\in W_2^s(0,1)$ and $g\in W_2^t(0,1)$. Developing $\exp\{\hat f(n)
\ri x\}$ into the Taylor series and changing the summation order (which
is allowed since the resulting double series converges absolutely), we
get
\[
    \Phi(f,g) = \sum_{k=1}^\infty\frac{(\ri x)^k}{k!}
                \mathrm{V.p.}
                \sum_{n\in\bZ} \hat g(n) \hat f^k(n)\re^{2\pi \ri nx}
              = \sum_{k=1}^\infty \frac{(\ri x)^k}{k!}\, h_k(x),
\]
where $h_k:= g\ast f^{<k>}$, $f^{<1>}:=f$, and $f^{<k>}$ for $k\ge2$ is
the~$k$-fold convolution of~$f$ with itself. By
Proposition~\ref{pro:B.conv}, the functions $h_k$ belong to
$W_2^{s+t}(0,1)$, and
\[
    \|h_k\|_{s+t} \le C^{k+1}\|f\|^k_s\|g\|_t
\]
with the constant $C$ of that proposition. Also, the operator~$M$ of
multiplication by~$\ri x$ is continuous in every $W_2^r(0,1)$, $r\in
[0,2]$; denoting by~$\|M\|_r$ the norm of the operator~$M$
in~$W_2^r(0,1)$, we conclude that $\Phi(f,g)$ belongs to
$W_2^{s+t}(0,1)$ and
\[
    \|\Phi(f,g)\|_{s+t}\le C \|g\|_t
        \bigl( \exp\{C\|M\|_{s+t}\|f\|_s\}-1\bigr).
\]
Moreover, every summand $(\ri x)^k h_k /k!$ depends analytically on $f$ and $g$; since the series converges absolutely in~$W_2^{s+t}(0,1)$, analyticity of~$\Phi$ follows.

By similar arguments we also find that
\[
    \Phi(f_1,g) - \Phi(f_2,g)
        = \sum_{k=1}^\infty \frac{(\ri x)^k}{k!}\,
            g\ast (f_1^{<k>}-f_2^{<k>});
\]
since
\[
    \|f_1^{<k>}-f_2^{<k>}\|_{s} \le kC^k\|f_1-f_2\|_s
        \bigl(\|f_1\|_s + \|f_2\|_s\bigr)^{k-1},
\]
it follows that
\[
    \|\Phi(f_1,g)-\Phi(f_2,g)\|_{s+t}
        \le C^2 \|M\|_{s+t}\|f_1-f_2\|_s\|g\|_t
            \exp\{C\|M\|_{s+t}(\|f_1\|_s + \|f_2\|_s)\}.
\]

We observe now that the mapping~$\Phi$ is linear in the second
argument; therefore,
\begin{align*}
    \|\Phi(f_1,g_1)-\Phi(f_2,g_2)\|_{s+t}
        &\le   \|\Phi(f_1,g_1)-\Phi(f_2,g_1)\|_{s+t}
            + \|\Phi(f_2,g_1-g_2)\|_{s+t}\\
        &\le C_1\{\|f_1-f_2\|_s + \|g_1-g_2\|_t\},
\end{align*}
where $C_1\le C^2 \|M\|_{s+t}(1+K)\exp\{C\|M\|_{s+t}K\}$ as long as
\[
    \|f_1\|_s + \|f_2\|_s  + \|g_1\|_t + \|g_2\|_t \le K,
\]
and the desired Lipschitz continuity of $\Phi$ follows.
\end{proof}

We observe that
\[
    \Phi(f,g) + g = \mathrm{V.p.}\sum_{n\in\bZ}
        \hat g(n) \exp\{\hat f(n)\ri x\} \re^{2\pi\ri nx}
\]
and hence the mapping
\[
    (f,g) \mapsto \mathrm{V.p.}\sum_{n\in\bZ}
        \hat g(n) \exp\{\hat f(n)\ri x\} \re^{2\pi\ri nx}
\]
is uniformly continuous from $W_2^s(0,1)\times W_2^t(0,1)$ into
$W_2^t(0,1)$.

Also, in the definition of $\Phi$, we can formally take $g$ to be the
unity $\delta$ of the convolution algebra~$L_2(0,1)$, which results
in~$\hat g(n)\equiv 1$. Slightly adapting the above proof, we get the
following result.

\begin{lemma}\label{lem:C.sum2}
Fix $s\in[0,1]$ and, for $f\in W_2^s(0,1)$, set
\[
    h(f) := \mathrm{V.p.}\sum_{n\in\bZ}
            [\exp\{\hat f(n)\ri x\}-1] \re^{2\pi\ri nx}.
\]
Then the function $h(f)$ belongs to~$W_2^s(0,1)$ and the mapping
$f\mapsto h(f)$ is analytic in~$W_2^s(0,1)$ and Lipschitz continuous
on bounded subsets.
\end{lemma}

We say that a function $f$ is \emph{odd} (resp. \emph{even}) over
$(0,1)$ (with respect to~$\tfrac12$) if it satisfies the relation $f(1-x)=-f(x)$ (resp., the
relation $f(1-x) = f(x)$) a.e.\ on~$(0,1)$. For every integrable
function $f$ we define its odd part $f_{\mathrm{odd}}$ and even
part~$f_{\mathrm{even}}$ by the equalities
\[
    f_{\mathrm{odd}}(x) := \frac{f(x)-f(1-x)}2, \qquad
    f_{\mathrm{even}}(x) := \frac{f(x)+f(1-x)}2.
\]
Since the reflection operator is continuous in every space $W_2^s(0,1)$,
$s\ge0$, the mappings $f\mapsto f_{\mathrm{odd}}$ and $f\mapsto
f_{\mathrm{even}}$ are bounded in every $W_2^s(0,1)$. We denote by
$W^s_{2,\mathrm{odd}}(0,1)$ and $W^s_{2,\mathrm{even}}(0,1)$ the
subspaces of $W_2^s(0,1)$ consisting of functions that are respectively
odd and even over $(0,1)$.

\begin{corollary}\label{cor:C.Phi}
Fix $s$ and $t$ in~$[0,1]$. Then the mappings
\[
    (f,g) \mapsto \sum_{n\in\bN}
        c_{2n}(g) \cos [2\pi nx + 2s_{2n}(f)x]
\]
from $W_{2,\mathrm{odd}}^s(0,1)\times W_{2,\mathrm{even}}^t(0,1)$
into $W_2^{t}(0,1)$ and
\[
    f \mapsto \sum_{n\in\bN}
       \{\cos [2\pi nx + 2s_{2n}(f)x] - \cos2\pi nx\}
\]
from $W_{2,\mathrm{odd}}^s(0,1)$ into $W_2^s(0,1)$ are analytic and
Lipschitz continuous on bounded subsets.
\end{corollary}

Indeed, for an odd~$f$ we have $\hat f(n) = -\hat f(-n)$ and, as a
result,
\[
    s_{2n}(f) = \frac1{2\ri}\bigl[\hat f(-n)-\hat f(n)\bigr]
              = \ri \hat f(n);
\]
for an even $g$, we similarly have $c_{2n}(g)=\hat g(n)= \hat g(-n)$. Therefore the
two series above coincide with $\tfrac12\Phi(2\ri f,g)+\tfrac12g$ and the
function~$\tfrac12h(2\ri f)$ of Lemma~\ref{lem:C.sum2} respectively, and the
result follows.

\begin{lemma}\label{lem:C.Psi}
For $f$ and $g$ in~$L_2(0,1)$, set
\[
    \Psi(f,g):= \mathrm{V.p.}\sum_{n\in\bZ}
        (-1)^n \int_0^1 g(t) \exp\{[\pi n+ \hat f(n)]\ri(1-2t)\}\,dt\, \re^{2\pi\ri nx};
\]
then for every $s\in[0,1]$ the mapping
\[
    \Psi\,:\, L_2(0,1)\times W_2^s(0,1) \to W_2^{s}(0,1)
\]
is analytic and Lipschitz continuous on bounded subsets.
\end{lemma}

\begin{proof}
The coefficient of $\re^{2\pi\ri nx}$ in the above series for~$\Psi$
can be written as
\begin{equation}\label{eq:C.c2n}
    \int_0^1 g(t) \exp\{\hat f(n) \ri(1-2t)\} \re^{-2\pi\ri nt}\,dt
\end{equation}
and gives the $n$-th Fourier coefficient of the function
 \(
    h:=\sum_{k=0}^\infty {h_k}/{k!},
 \)
with $h_0:=g$, $h_k:= f^{<k>} \ast M_1^k g$ for $k\ge1$, and $M_1$
being the operator of multiplication by the function~$\ri(1-2t)$,
i.e., $\Psi(f,g)=h$. Since $h_k$ belongs to $W_2^s(0,1)$ and its
norm there obeys the estimate
\[
    \|h_k\|_s \le C^k(\|f\|_0)^k\|M_1\|_s^k\|g\|_s
\]
with $C$ being the constant of Proposition~\ref{pro:B.conv} and
$\|M_1\|_s$ denoting the norm of the operator $M_1$ in the
space~$W_2^s(0,1)$, we conclude that the mapping $\Psi$ is analytic.
Its Lipschitz continuity on bounded subsets is established in the usual
manner.
\end{proof}

As above, by taking an odd~$f$ and an even~$g$ of zero mean, we see
that the expressions of~\eqref{eq:C.c2n} for $n$ and $-n$ coincide with
\begin{equation}\label{eq:C.2}
    2(-1)^n \int_0^1 g(t) \cos\{[\pi n +s_{2n}(-\ri f)](1-2t)\}\,dt
\end{equation}
and give the $2n$-th cosine Fourier coefficient of the
function~$h=:\Psi(f,g)\in W^s_{2,\mathrm{even}}(0,1)$.


\end{document}